\documentclass[reqno,11pt]{amsart}
\usepackage{amsthm,amsfonts,amssymb,euscript,
mathrsfs,graphics,color,amsmath,latexsym,marginnote}
\usepackage{dsfont}   
  
\theoremstyle{plain}

\usepackage{hyperref}
\usepackage{times}
\usepackage{mathtools}

\numberwithin{equation}{section}

\newtheorem{theorem}{Theorem}[section]
\newtheorem{proposition}[theorem]{Proposition}
\newtheorem{lemma}[theorem]{Lemma}

\newtheorem{remark}[theorem]{Remark}
\newtheorem{remarks}[theorem]{Remark}

\newtheorem{definition}[theorem]{Definition}
\newcommand{\be}{\begin{equation}}
\newcommand{\ee}{\end{equation}}

\newcommand{\uno}{\mathds{1}}
\newcommand{\e}{\varepsilon}
\newcommand{\ep}{\epsilon}

\newcommand{\ov}{\overline}

\newcommand{\R}{\mathbb R}
\newcommand{\C}{\mathbb C}

\newcommand{\Z}{\mathbb Z}

\newcommand{\N}{\mathbb N}
\newcommand{\T}{\mathbb T}

\newcommand{\s }{\sigma }
\newcommand{\ii }{{\rm i} }

\newcommand{\x }{\xi }

\newcommand{\pa}{\partial}

\newcommand{\opbw}{{Op^{{bw}}}}

\def\hat{\widehat}

\def\bar{\overline}

\def\cal{\mathcal}

\def\ba{\begin{aligned}}
\def\ea{\end{aligned}}
\def\beginm{\begin{multline}}
\def\endm{\end{multline}}

\providecommand{\vect}[2]{{\bigl[\begin{smallmatrix}#1\\#2\end{smallmatrix}\bigr]}}   
\providecommand{\sm}[4]{{\bigl[\begin{smallmatrix}#1&#2\\#3&#4\end{smallmatrix}\bigr]}}

\makeatletter

\def\l@subsection{\@tocline{2}{0pt}{2.5pc}{5pc}{}}
\def\l@subsubsection{\@tocline{3}{0pt}{4.5pc}{5pc}{}}
\renewcommand\tocchapter[3]{%
  \indentlabel{\@ifnotempty{#2}{\ignorespaces#2.\quad}}#3%
}
\newcommand\@dotsep{4.5}
\def\@tocline#1#2#3#4#5#6#7{\relax
  \ifnum #1>\c@tocdepth 
  \else
    \par \addpenalty\@secpenalty\addvspace{#2}%
    \begingroup \hyphenpenalty\@M
    \@ifempty{#4}{%
      \@tempdima\csname r@tocindent\number#1\endcsname\relax
    }{%
      \@tempdima#4\relax
    }%
    \parindent\z@ \leftskip#3\relax \advance\leftskip\@tempdima\relax
    \rightskip\@pnumwidth plus1em \parfillskip-\@pnumwidth
    #5\leavevmode\hskip-\@tempdima{#6}\nobreak
    \leaders\hbox{$\m@th\mkern \@dotsep mu\hbox{.}\mkern \@dotsep mu$}\hfill
    \nobreak
    \hbox to\@pnumwidth{\@tocpagenum{#7}}\par
    \nobreak
    \endgroup
  \fi}
\def\l@subsection{\@tocline{2}{0pt}{2.5pc}{5pc}{}}

\begin{document}
\bibliographystyle{plain}

\title[Growth of Sobolev norms for DNLS on generic tori]{Quadratic lifespan and
 growth of Sobolev norms for derivative Schr\"odinger equations on generic tori}

\date{}


\author{Roberto Feola}
\address{\scriptsize{Dipartimento di Matematica, Universit\`a degli Studi di Milano, Via Saldini 50, I-20133}}
\email{roberto.feola@unimi.it}

\author{Riccardo Montalto}
\address{\scriptsize{Dipartimento di Matematica, Universit\`a degli Studi di Milano, Via Saldini 50, I-20133}}
\email{riccardo.montalto@unimi.it}

\thanks{ {\em Acknowledgements.} Riccardo Montalto is supported by INDAM-GNFM. The authors warmly thank Dario Bambusi for many useful discussions and comments.}  
 



\begin{abstract}   
We consider a family of Schr\"odinger equations with unbounded Hamiltonian quadratic 
nonlinearities on a generic tori of dimension $d\geq1$.
We study the behaviour of high Sobolev norms $H^{s}$, $s\gg1$, of solutions
with initial conditions in $H^{s}$ whose 
 $H^{\rho}$-Sobolev norm, $1\ll\rho\ll s$, is smaller than $\e\ll1$.
We provide a control of the $H^{s}$-norm over a time interval of order $O(\e^{-2})$.

Due to the lack of conserved quantities controlling high Sobolev norms,
the key ingredient of the proof is the construction of a modified energy 
equivalent to 
the ``low norm'' $H^{\rho}$ (when $\rho$ is sufficiently high) over a nontrivial time interval
$O(\e^{-2})$.
This is achieved by means of normal form techniques for quasi-linear equations
involving para-differential calculus.
The main difficulty
is to control the possible loss of derivatives due to the small divisors 
arising form  three waves interactions.
By performing ``tame'' energy estimates we obtain 
upper bounds for higher Sobolev norms $H^{s}$.

\end{abstract}

\maketitle
\tableofcontents

\section{Introduction}
We consider the derivative Schr\"odinger equation (DNLS)
\begin{equation}\label{main}
 \partial_t u = \ii (\Delta_g u - m u  -  {\cal Q}(u, \bar u))\,,\quad u=u(t,x)\,,\;\;
 x\in \mathbb{T}^{d}:=(\mathbb{R}/2\pi\mathbb{Z})^{d}\,,\quad d\geq1\,,
\end{equation}
where $m > 0$ is the mass, the operator $\Delta_{g}$ is defined by linearity as
\begin{equation}\label{matmetrica}
\Delta_{g}e^{\ii j\cdot x}=-\|j\|_{g}^{2}e^{\ii j\cdot x}\,,\quad
\|j\|_{g}^{2}:=G j\cdot j\,,\quad  j\in \mathbb{Z}^{d}\,,
\end{equation}
with $G = (g_{i j})_{i, j = 1, \ldots, d}$  a strictly positive definite, symmetric, matrix, i.e. $G \xi \cdot \xi \geq c_0 |\xi|^2$ 
for any $\xi \in \Z^d \setminus \{ 0\}$, and $c_0 > 0$. 
The nonlinearity  $Q(u,\bar{u})$ has the form
\begin{equation}\label{nonlinQQQ}
Q(u,\bar{u}):=(\pa_{\bar{u}}f)(u,\nabla u)-\sum_{i=1}^{d}\pa_{x_{i}} (\pa_{\ov{u_{x_i}}}f)(u,\bar{u})\,,
\end{equation}
where we denoted $\pa_{u}:=(\pa_{{\rm Re}(u)}-\ii \pa_{{\rm Im}(u)})/2$ and  
$\pa_{\bar{u}}:=(\pa_{{\rm Re}(u)}+\ii \pa_{{\rm Im}(u)})/2$
the Wirtinger derivatives
and where
$f(y_0,y_1,\ldots,y_{d})\in  C^{\infty}(\mathbb{C}^{d+1};\mathbb{R})$ 
(in the \emph{real} sense, i.e. $f$ is $C^{\infty}$
  as function of ${\rm Re}(y_i)$, ${\rm Im}(y_i)$) 
  is a homogenenous polynomial of degree $3$ 
satisfying 
\begin{equation}\label{derord2}
\pa_{y_{i}}\pa_{\ov{y_{j}}}f=\pa_{\ov{y_{i}}}\pa_{\ov{y_{j}}}f\equiv0\,,
\quad \forall\, i,j=1,\ldots,d\,,
\quad \forall\,
(y_0,y_1,\ldots,y_d)\in \mathbb{C}^{d+1}\,.
\end{equation}
Thanks to \eqref{nonlinQQQ} one can note that 
the equation is Hamiltonian, namely \eqref{main} can be written as
\[
\partial_t u = - \ii \nabla_{\bar u} H(u, \bar u)\,,
\qquad H (u, \bar u) := \int_{\T^d} 
(\Lambda u)\cdot\bar{u}\,dx
+ \int_{\T^d} f(u, \bar u)\, d x\,,
\]
where\footnote{ 
$\nabla_{u}:=(\nabla_{{\rm Re}(u)}-\ii \nabla_{{\rm Im}(u)})/2$ and  
$\nabla_{\bar{u}}:=(\nabla_{{\rm Re}(u)}+\ii \nabla_{{\rm Im}(u)})/2$, 
$\nabla$ denotes the $L^{2}$-gradient.}
$\Lambda:=\Lambda(D):=-\Delta_{g}+m$ is the Fourier multiplier
with symbol
\begin{equation}\label{def Lambda xi}
\Lambda(\x):=\|\x\|^{2}_{g}+m:=G \x\cdot \x+m\,,\quad \x\in \mathbb{Z}^{d}\,.
\end{equation}
The main result of the paper is the following.

\begin{theorem}\label{teo principale}
For almost every $m \in (0, + \infty)$ the following holds.
There exists $\rho \gg 0$ large enough and $\e \equiv \e(\rho) \ll 1$ small enough such 
that for any initial datum $u_0 \in H^{\rho}(\mathbb{T}^{d};\mathbb{C})$, 
$\| u_0 \|_{H^{\rho}} \leq \e$, 
there exists a unique solution 
$u \in C^0\big( [- T_\rho, T_\rho], H^\rho(\mathbb{T}^{d};\mathbb{C}) \big)$  
of \eqref{main} with 
$u(0, \cdot) = u_0$, with
\begin{equation}\label{stimathm1}
\| u(t) \|_{H^{\rho}} \lesssim_\rho \e\,,\quad \forall\, t\in[-T_{\rho}, T_{\rho}]\,,\quad
T_\rho := \mathtt c(\rho) \e^{- 2}
\end{equation}
for some $\mathtt{c}(\rho)\leq1$.
Moreover, assume in addition that $u_0 \in H^s(\mathbb{T}^{d};\mathbb{C})$, 
$s \geq \rho$ (without any smallness assumption of $\| u_0 \|_{H^{s}}$). 
Then $u \in C^0\big( [- T_{\rho}, T_{\rho}], H^s(\mathbb{T}^{d};\mathbb{C}) \big)$ of \eqref{main} 
which remains bounded on $[- T_{\rho},  T_{\rho}]$, namely
\begin{equation}\label{stimathm2}
\| u(t) \|_{H^{s}} \lesssim_s \| u_0 \|_{H^{s}}\,, \quad \forall t \in [- T_{\rho}, T_{\rho}]\,. 
\end{equation}
\end{theorem}
\noindent
Some comments on the result above are in order.

\vspace{0.2em}
First notice that equation \eqref{main} can be seen as a nonlinear Schr\"odinger equation
posed on a torus  $\mathbb{T}^{d}_{\Gamma}:=\mathbb{R}^{d}/\Gamma$ with 
arbitrary periodicity lattice $\Gamma$.
In view of the assumptions in \eqref{nonlinQQQ}-\eqref{derord2}, we have that 
$Q(u,\bar{u})$ is a Hamiltonian nonlinearity containing \emph{at most}
one spatial derivative of the unknown $u(t,x)$.
Hence, as far as we know, Theorem \ref{teo principale}
provides the first long existence results of solutions for nonlinear Schr\"odinger equations 
with derivatives and on a manifold different from the \emph{square} torus. 
 
We also could consider nonlinearities of order $m$ with $m < 2$ 
but we preferred to write the paper for nonlinearities depending on $\nabla u$, 
since they are more physical (it is basically the case of magnetic potentials). 

The bound \eqref{stimathm1} shows indeed that solutions
evolving from sufficiently regular initial data fo size $\e\ll1$
remain small 
over a time interval of size $O(\e^{-2})$. This lifespan which is strictly larger than
the time of existence provided by local theory which is of order $O(\e^{-1})$.
In addition to this our result provide a control, on the same time interval, of the growth
of high Sobolev norms of solutions of \eqref{main}.
The bound \eqref{stimathm2} actually shows that 
the $H^{s}$-Sobolev norm of a solution remains \emph{bounded}
by only requiring a smallness conditions on a \emph{low} norm 
of the initial datum.
The second part of the Theorem is a consequence of sharp tame \emph{\`a priori} 
estimates on the solutions. Unfortunately the equation has no conserved quantities 
which control for every time the Sobolev norms $H^{\rho}$
with $\rho>1$. Hence we are only able to obtain the bound \eqref{stimathm2}
over a long (but finite) time interval, by constructing a modified energy for the $H^{\rho}$-norm
with normal forms techniques.
The result we obtained is in the same spirit of \cite{DelMa}
by Delort-Masmoudi. 

We require the Hamiltonian assumption on the nonlinearity
in order to guarantee the well-posedness of the Cauchy problem associated to
\eqref{main} at least for short time. Actually this hypothesis could be weakened.
For more details, we refer for instance to the introduction of \cite{Feola-Iandoli-Loc}.

We also remark that the mass parameter $m>0$ in \eqref{def Lambda xi}
will be used to provide suitable lower bounds on  \emph{three wave interactions}.

\vspace{0.5em}
\noindent
{\bf Some related literature.}
We now present some known results on the long time existence and  stability 
for derivative Schr\"odinger equations.

\vspace{0.2em}
\noindent
\emph{Local well-posedness.}
Many authors considered 
equations of the type \eqref{main} (even without the assumption \eqref{derord2}) 
in the \emph{Euclidean case} (i.e. when $x\in \mathbb{R}^{d}$).
The first existence result is due to Poppenberg in \cite{Pop1} for a special model in one dimension,
later extended by Colin \cite{colin} to any dimension.
A more general class of quasilinear Schr\"odinger equation is studied in the pioneering work of
Kenig-Ponce-Vega \cite{KPV}. We also mention a recent paper \cite{MMT3} by 
Marzuola-Metcalfe-Tataru  (see also references therein) which optimize the result in \cite{KPV}
in terms of the regularity of the initial data.
The situation drastically changes when the equation is posed on a compact manifold.
Indeed, 
Christ in \cite{Cris} provides examples  of Schr\"odinger equations with derivatives  
which are ill-posed on the 
circle $\mathbb{S}^1$ and well posed on $\R$. 
We mention that a local existence 
has been obtained 
on the circle 
 by Baldi-Haus-Montalto in \cite{BHM}  
and by Feola-Iandoli  in \cite{Feola-Iandoli-Loc} with different techniques.
In \cite{Feola-Iandoli-local-tori} the authors extend the latter results to any 
\emph{squared} $d$-dimensional tori. 
Our assumption in \eqref{nonlinQQQ} on the nonlinearity guarantees 
that the local well-posedness for \eqref{main}
can be obtained in the same spirit of \cite{Feola-Iandoli-local-tori}.

\vspace{0.2em}
\noindent
\emph{Global well-posedness.}
All the aforementioned results regards the \emph{local in time} well-posedness for 
quasilinear Schr\"odigner
equations. The global well-posedness has been established on $\R^2$ and $\R^3$ 
by de Bouard-Hayashi-Saut \cite{saut-glob} 
in dimension two and three for small data on  a model quasilinear Schr\"odinger.
In \cite{saut-glob}, dispersive properties of the flow are exploited in order to obtain 
a control of the Sobolev norms for long time. 
We also mention the paper \cite{MuPu} by Murphy-Pusateri
about the almost global existence for a 
non-gauge-invariant cubic nonlinear Schr\"odinger equation on $\mathbb{R}$.

\vspace{0.2em}
\noindent
\emph{Long time regularity and normal forms.}
On tori (or more in general on compact manifolds) 
there are no dispersive effects that could help in controlling the behaviour of 
the solutions
for long times. In order to extend the lifespan of solution
we use the powerful tool of normal form theory.
This approach has been successfully and widely used 
in the past starting form the study of semi-linear PDEs.
Without trying to be exhaustive we quote
Bourgain \cite{Bou96}, Bambusi \cite{Bam03} 
and Bambusi-Gr\'ebert \cite{BG}
where the authors
considered the Klein-Gordon equation on the circle.
They proved almost global existence in the sense that, for any $N\geq1$ and any initial datum in $H^{s}(\mathbb{T})$
of size $\e\ll1$ with $s\gg1$ large enough, the solution exist and its $H^{s}$-Sobolev norm
remains small over a time interval of size  $O(\e^{-N})$.
Similar results have  been obtained for semilinear equations also in higher space dimension.
We refer, for instance, to \cite{BDGS} 
by Bambusi-Delort-Gr\'ebert-Szeftel which considered PDEs on Zoll manifolds
(see also \cite{FG,DelortSzeft1}).
Normal form theory for 
quasi-linear equations have been constructed more recently. We quote 
Delort \cite{Delort-circle, Delort-sphere} for the Klein-Gordon on $\mathbb{S}^d$ and
Berti-Delort \cite{BD} for the gravity capillary water waves on $\mathbb{T}$.
For equations like \eqref{main}
 we mention \cite{Feola-Iandoli-Long, Feola-Iandoli-Totale} where it is exploited 
 the fact that (following the ideas of \cite{BD}) quasi-linear Schr\"odinger 
equations may be   reduced to constant coefficients 
trough a \emph{para-composition} generated by a diffeomorphism of the circle.

\vspace{0.2em}
\noindent
\emph{Normal forms on irrational tori.}
All the papers mentioned above have in common that the spectrum of the
linearized problem at zero has ``good separation'' properties.
This fact  depends on the geometry of the eigenvalues of the Laplace-Beltrami operator.
On irrational tori, for instance, differences of eigenvalues can accumulate to zero.
In this case, one typically gets very weak lower bounds on ``small divisors'' arising from
$n$-waves interactions (see Appendix \ref{sezione non rionanza}). 
The same problem occurs 
for the Klein Gordon equation posed on $\mathbb{T}^{d}$, $d\geq2$.
In dealing with this problem is out of reach (at the moment), 
but nevertheless one can obtain partial results.
We refer to Delort \cite{Delort-Tori},
 Fang and Zhang \cite{fang}, 
Zhang \cite{Zhang} for the Klein-Gordon, 
 Imekraz in \cite{Im} for the Beam equation on $\mathbb{T}^{2}$
 and
 Feola-Gr\'ebert-Iandoli \cite{FGI20}.
In this last case a special class of quasi linear Klein Gordon equation is considered. 
We finally quote  the remarkable work on multidimensional periodic water wave
by Ionescu-Pusateri \cite{IPtori}.

\vspace{0.2em}
\noindent
\emph{Growth of Sobolev norms for PDEs on tori.} For linear Schr\"odinger equations 
with time dependent potentials on tori, 
there are several results providing un upper bound $t^\e$ 
for the high Sobolev norms of the solutions. 
On rational tori $\T^d$, we mention the results 
of Bourgain \cite{Bou}, \cite{Bou2} and Delort \cite{Delort}. 
This results have been extended on the irrational torus 
by Berti and Maspero in \cite{BM19}. 
In these aforementioned results, the potential is bounded and the proof basically 
relies on the so called {\it Bourgain Lemma}. 
For Schr\"odinger equations on irrational tori with unbounded potentials 
(of order strictly smaller than $2$), the upper bound $t^\e$ on the growth 
of Sobolev norms has been proved in  \cite{BLM-growth}. 
The proof relies on a Pseudo-differential normal form 
and on a careful analysis of the resonant vector field, 
by showing that the flow generated by it is uniformly bounded in time. 

\noindent
For nonlinear Schr\"odinger equations on tori, 
by completely different methods, Bourgain \cite{Bourgain-libro} 
proved an upper bound $t^s$ for the $H^s$-norm of the solutions of 
nonlinear Schr\"odinger equations on $\T^2$. 
This result has been also generalized on more general manifolds 
in \cite{PTV} (see also references therein).

\vspace{0.5em}
\noindent
{\bf Plan of the paper and scheme of the proof.}
In the remaining part of the introduction, we briefly explain the strategy of our proof. 

\noindent
In order to prove a time of existence of size $O(\e^{- 2})$, 
we need to perform one step of normal form, in order to remove the quadratic terms.  
Hence in Section \ref{sez def simpboli e smoothing} 
we consider symbols which are sums of symbols linear in $u, \overline u$ 
plus symbols which are quadratic in $(u, \overline u)$. 
Similarly we define classes of smoothing operators. 
Since one is able to impose only very weak lower bounds on the three wave interactions 
(cf. Section \ref{sezione non rionanza}), 
the normal form procedure requires to use paradifferential calculus. 
In Section \ref{sez forma normale paradiff}, we construct a change of variables 
$u = \Phi(u)[w]$ ($u(t, x)$ is a smooth solution of \eqref{main} 
defined on a time interval $[- T, T]$) which transforms the equation \eqref{main} 
into another one which has the form 
\[
\partial_t w +\ii (- \Delta_g + m) w + \ii {\rm Op}^{bw}(z(u; x, \xi))w + {\cal R}(u) w = 0
\]
where ${\cal R}(u)$ is a smoothing remainder, 
i.e. $\| {\cal R}(u) w \|_{H^{s + N}} \lesssim_{s, N} \| u \|_{H^{\rho}} \| w \|_{H^{s}}$ for $N \gg 0$, 
$\rho \gg N$, $s \gg N$ and the {\it normal form symbol} $z(U; x, \xi)$ 
is real and it has the property that its Fourier tarnsform $\widehat z(U; k, \xi)$ is non zero if 
\[
 |(\xi; k)| \leq 
\langle \xi  \rangle^{\delta}| k|^{-\tau} \textrm{ and } |k |\leq \langle \xi \rangle^{\ep}\,
\]
cf. Definition \ref{simboli forma normale}. 
This normal form step is essentially a {\it nonlinear analogue} 
of the method developed in \cite{comm-pdes}, \cite{BLM-growth} at a linear level. 

\noindent
At this point, in Section \ref{sec:BNFstep}, 
we perform a Poincar\'e-Birkhoff normal form step in order 
to remove the quadratic terms from the smoothing remainder ${\cal R}(u) w$. 
The loss of derivatives in the estimates of the three wave interactions is then compensated 
by the fact that the remainder is smoothing. In \cite{BLM-growth}, 
it is proved that the flow associated to normal form symbols 
are well defined on $H^s$ and uniformly bounded in time. 
This fact allows in Section \ref{sezione finale stime energia} 
to perform an energy estimate which shows that  
$\| u(t ) \|_{H^{\rho}} \lesssim_\rho \e$ for $t \in [- T_\rho, T_\rho]$ with $T_\rho = O(\e^{- 2})$, 
for some $\rho \gg 1$ large enough, provided the initial datum 
$\| u_0 \|_{H^{\rho}} \leq \e$ is small enough. 
If in addition, the initial datum $u_0 \in H^s$, with $s > \rho$ 
(but with no smallness assumption on $\| u_0 \|_{H^{s}}$), 
a bootstrap argument shows that 
$\| u(t) \|_{H^{s}} \lesssim_s \| u_0 \|_{H^{s}}$ for any $t \in [- T_\rho, T_\rho]$, 
implying that there is no growth of {\it high Sobolev norms} 
over the time interval $[- T_\rho, T_\rho]$. 

\noindent
We finally remark that if one considers the equation \eqref{main} 
with the standard Laplacian, there are no small divisors since if the mass is not an integer, 
the three wave interactions are bounded from below by a constant. 
On the other hand, generically 
(meaning for a generic choice of the matrix $G$ in \eqref{matmetrica}), 
the three wave interactions accumulate to zero and one 
is able to prove only very weak non-resonance conditions, 
see Appendix \ref{sezione non rionanza} for more details.

%

\section{Functional setting}
We denote by $H^{s}(\mathbb{T}^d;\mathbb{C})$
(respectively $H^{s}(\mathbb{T}^d;\mathbb{C}^{2})$)
the usual Sobolev space of functions $\mathbb{T}^{d}\ni x \mapsto u(x)\in \mathbb{C}$
(resp. $\C^{2}$).
We expand a function $ u(x) $, $x\in \mathbb{T}^{d}$, 
 in Fourier series as 
\begin{equation*}
u(x) = \frac{1}{(2\pi)^{{d}/{2}}}
\sum_{n \in \Z^{d} } \hat{u}(n)e^{\ii n\cdot x } \, , \qquad 
\hat{u}(n) := \frac{1}{(2\pi)^{{d}/{2}}} \int_{\mathbb{T}^{d}} u(x) e^{-\ii n \cdot x } \, dx \, .
\end{equation*}
We endow $H^{s}(\mathbb{T}^{s};\mathbb{C})$ with the norm
\begin{equation*}
\|u\|^{2}_{s}: =\|u\|_{H^{s}}^{2}:=(\langle D\rangle^{s}u,\langle D\rangle^{s} u)_{L^{2}}\,, 
\qquad
\langle D\rangle e^{\ii j\cdot x}=\langle j\rangle  e^{\ii j\cdot x}\,,\;\;\; 
\forall \, j\in \mathbb{Z}^{d}\,,
\end{equation*}
where $\langle j\rangle:=\sqrt{|j|^{2}+1}$ and $(\cdot,\cdot)_{L^{2}}$ denotes the standard complex $L^{2}$-scalar product
 \begin{equation}\label{scalarL}
 (u,v)_{L^{2}}:=\int_{\mathbb{T}^{d}}u\cdot\bar{v}dx\,, 
 \qquad \forall\, u,v\in L^{2}(\mathbb{T}^{d};\mathbb{C})\,.
 \end{equation}
For $U=(u_1,u_2)\in H^{s}(\mathbb{T}^d;\mathbb{C}^{2})$
we just set
$\|U\|_{s}=\|u_1\|_{s}+\|u_2\|_{s}$.

\vspace{0.2em}
  \noindent
{\bf Notation}. 
We shall 
use the notation $A\lesssim B$ to denote 
$A\le C B$ where $C$ is a positive constant
depending on  parameters fixed once for all, for instance $d$
 and $s$.
 We will emphasize by writing $\lesssim_{q}$
 when the constant $C$ depends on some other parameter $q$.
 To shorten the notation we shall write $H^{s}=H^{s}(\mathbb{T}^{d};\mathbb{C})$.

\subsection{Classes of symbols and operators}\label{sez def simpboli e smoothing}
In this section we introduce symbols and operators we shall use along the paper. We follow the notation of
\cite{Feola-Iandoli-local-tori} but with symbols  introduced in \cite{comm-pdes}. 

Given a symbol $a(x, \xi)$ of order $m$, and fixing $\delta \in (0, 1)$ (very close to one) we define for any $s \in \N$, the norm $|a|_{m, s}$ as 
\begin{equation}\label{seminormSimbo}
|a|_{m, s} := \sup_{|\alpha_1| + |\alpha_2| \leq s} 
\sup_{(x, \xi) \in \T^d \times \R^d} |\partial_x^{\alpha_1} \partial_\xi^{\alpha_2} a(x, \xi) 
\langle \xi \rangle^{- m + \delta |\alpha_2|}|\,,
\end{equation}
and we define ${\cal N}^{m}_s$ the space of the ${C}^s$ functions 
$(x, \xi) \mapsto a(x, \xi)$ such that $|a|_{m, s} < \infty$. 
If the symbol $a(x, \xi)$ is independent of $x$, 
namely it is a Fourier multiplier $a(\xi)$, then the norm is given by 
\[
|a|_{m, s} := \sup_{|\alpha|\leq s} \sup_{ \xi \in  \R^d} | \partial_\xi^{\alpha} a( \xi) \langle \xi \rangle^{- m + \delta |\alpha|}|\,. 
\]
The following elementary lemma holds. 
\begin{lemma}\label{non hom symbols Fourier}
Let $m \in \R$, $N, s \in \N$, $a \in {\cal N}^{m}_{s + N}$. 
Then for any $k \in \Z^d$,  
\[
\begin{aligned}
& |\widehat a( k, \cdot)|_{m, s} \lesssim_{ N} \langle k \rangle^{- N} |a|_{m, s + N}\,.
\end{aligned}
\]
\end{lemma}
\begin{proof}
By a simple integration by parts, one has 
\[
\begin{aligned}
 k_i^N  \widehat a(k, \xi)  
 & = - \frac{1}{(-\ii)^N}\frac{1}{(2 \pi)^{d/2}} \int_{\T^d}  
 \widehat a( x, \xi)\partial_{x_i}^N (e^{- \ii k \cdot x})\, d x 
 =  \frac{(- 1)^{N + 1}}{(-\ii)^N}\frac{1}{(2 \pi)^{d/2}} \int_{\T^d} \partial_{x_i}^N a( x, \xi)e^{- \ii k \cdot x}\, d x\,.
 \end{aligned}
\]
Hence 
\begin{equation*}
\langle k \rangle^N| \widehat a(k, \cdot) |_{m, s}  
\lesssim_N 
{\rm max}_{i = 1, \ldots, d} |\partial_{x_i}^N a|_{m, s}  \lesssim_N |  a|_{m, s + N}
\end{equation*}
and the claimed statement has been proved. 
\end{proof}

\noindent
{\bf The Bony-Weyl quantization.}
Let $0<\epsilon< 1/2$ and consider
a smooth function $\eta : \mathbb{R}\to[0,1]$
\begin{equation*}
\eta(\x)
=\left\{
\begin{aligned}
&1 \quad {\rm if} |\x|\leq 5/4 
\\
&0 \quad {\rm if} |\x|\geq 8/5 
\end{aligned}\right.\qquad {\rm and \; define}\quad
\qquad \eta_{\epsilon}(\x):=\eta(|\x|/\epsilon)\,.
\end{equation*}
 For a symbol $a(x,\x)$ in $\mathcal{N}_{s}^{m}$
we define its (Weyl)
quantization as 
\begin{equation}\label{quantiWeyl}
\opbw(a)h:=\frac{1}{(2\pi)^{d}}\sum_{j\in \mathbb{Z}^{d}}e^{\ii j\cdot x}
\sum_{k\in\mathbb{Z}^{d}}
\eta_{\epsilon}\Big(\frac{|j-k|}{\langle j+k\rangle}\Big)
\widehat{a}\big(j-k,\frac{j+k}{2}\big)\widehat{h}(k)
\end{equation}
where $\widehat{a}(\eta,\x)$ denotes the Fourier transform 
of $a(x,\x)$ in the variable 
$x\in \mathbb{T}^{d}$.
\begin{remark}\label{rmk:simboloLambda}
Notice that the symbol $\Lambda(\x)$ in \eqref{def Lambda xi}
belongs to $\mathcal{N}_{s}^{\,2}$ for any $s\in \mathbb{R}$, 
with $|\Lambda|_{2,s}\lesssim_{s}1$.
Moreover (recall \eqref{matmetrica}) 
we have that $-\Delta_{g}+m=\opbw(\Lambda(\x))$.
\end{remark}

The following results follows by standard paradifferential calculus.
\begin{lemma}{\bf (Action of Sobolev spaces).}\label{standard cont bony}
Let $m \in \R$, $s_0 > d/2$. Then for any $s \geq 0$, the linear map 
\[
{\cal N}^m_{s_0} \to {\cal L}(H^{s + m}, H^s), \quad a \mapsto \opbw(a)
\]
is continuous, namely 
\[
\| {\rm Op}^{bw}(a) \|_{{\cal L}(H^{s + m}, H^s)} \lesssim_s |a|_{m, s_0}\,. 
\]
\end{lemma}
In the paper we shall deal with symbols in $\mathcal{N}_{s}^{m}$
depending nonlinearly on a function $u\in H^{s}(\mathbb{T}^{d};\mathbb{C})$.
Let us now introduce the spaces
\begin{equation}\label{hcic}
{\bf H}^{s}:=\Big( H^{s}(\mathbb{T}^{d};\C)\times H^{s}(\mathbb{T}^{d};\C) \Big)\cap\mathcal{U}\,,
\qquad
\mathcal{U}:=\{(u^{+},u^{-})\in \mathbb{L}^{2}(\mathbb{T}^{d};\mathbb{C}^{2})\; :\; \ov{u^{+}}=u^{-}\}\,.
\end{equation}
We denote by $B_s( r )$ the ball 
\[
B_s( r ) := \Big\{ U = (u, \bar u)\in {\bf H}^{s}
\;:\; \| U \|_s \leq r \Big\}\,. 
\]
\begin{definition}{\bf (Non-Homogeneous Symbols).}\label{simboli non lineari}
Let $m \in \R$, $p \in \N$. 
We say that a map $U = (u, \bar u )\mapsto a(U; x, \xi)$ 
belongs to the class $\Gamma^{m}_p$ if there exists $s_0 > 0$ 
such that for any $s \geq s_0$, 
there exists $r = r(s) \in (0, 1)$, $\sigma_s \gg s$ such that the map 
\[
B_{\sigma_s}(r) \to {\cal N}^m_s\,, \quad U \mapsto a(U; x, \xi)\,,
\]
is ${\cal C}^\infty$-smooth and vanishes at $U = 0$ of order $p$. 
\end{definition}
\begin{remark}{\bf (Estimates on non-homogenenous symbols).}\label{stime simboli in omogeneita}
Clearly by the latter definition, one has the following estimates.
\[
|a(U; \cdot)|_{m, s} 
\lesssim_s \| U \|_{\sigma_s}^p
\]
If $n \leq p$, $H_1, \ldots, H_n\in H^{\sigma_s}$, 
\[
|d^n a(U; \cdot)[H_1, \ldots, H_n]|_{m, s} 
\lesssim \| U \|_{\sigma_s}^{p - k} \| H_1 \|_{\sigma_s} \ldots \| H_n \|_{\sigma_s}\,. 
\]
If $n > p$, then 
\[
|d^n a(U; \cdot)[H_1, \ldots, H_n]|_{m, s} 
\lesssim  \| H_1 \|_{\sigma_s} \ldots \| H_n \|_{\sigma_s}\,. 
\]
\end{remark}

\begin{definition}{\bf (Linear symbols in $(u, \bar u)$).}\label{simboli lineari}
Let $m \in \R$. We say that a linear map $U = (u, \bar u )\mapsto a(U; x, \xi)$ 
belongs to the class $O^{m}_1$ if it is in the class 
$\Gamma^m_1$ and the symbol $a(U; x, \xi)$ is linear w.r.t. $U$, 
namely it has the form 
\[
a(U; x, \xi) = \sum_{k \in \Z^d,\s\in \{\pm\}} m_{\s}(k, \xi) \widehat{u}^{\s}(k) e^{\s\ii k \cdot x} 
\]
where, for any $k\in \mathbb{Z}^{d}$, we denoted
\[
\widehat{u}^{\s}(k)=\widehat{u}(k)\,,\;\;{\rm if}\;\;\s=+\,,\qquad   
\widehat{u}^{\s}(k)=\ov{\widehat{u}(k)}\,,\;\;{\rm if}\;\;\s=-\,.
\]
\end{definition}

\begin{remark}
Notice that one has the inclusion $O^m_1 \subseteq \Gamma^m_1$. 
\end{remark}

\begin{definition}{\bf (Symbols).}\label{simboli}
Given $m \in \R$, we say that a symbol $ a \in \Sigma^m_1$ if $a = a_l + a_q$ with $a_l \in O^m_1$ and $a_q \in \Gamma^m_2$. 
\end{definition}

\begin{definition}{\bf (Classes of para-differential operators).}\label{Op paradiff}
$(i)$ We say that a linear operator $A$ is in the class 
${\cal O}{\cal B}_\Gamma(m, p)$ 
if there exists $a \in \Gamma^m_p$ such that $A = {\rm Op}^{bw}(a)$. 

\noindent
$(ii)$ We say that a linear operator $A$ is in the class 
${\cal O}{\cal B}_O(m)$, if there exists $a \in O^m_1$ such that $A = {\rm Op}^{bw}(a)$. 

\noindent
$(iii)$ We say that a linear operator $A$ is in the class 
${\cal O}{\cal B}_\Sigma(m)$, if there exists $a \in \Sigma^m_1$ such that $A = {\rm Op}^{bw}(a)$.
\end{definition}

We now start by defining the classes of smoothing operators that we use in our procedure. 
\begin{definition}{\bf (Non-Homogenenous smoothing operators).}\label{nonomosmooth}
Let $N \in \N$. We say that a map $(U, w) \mapsto {\cal R}(U)[w] $ 
belongs to the class ${ \cal S}_2( N)$ if there exists $\rho \equiv \rho_N > N$ 
such that for any $s \geq \rho$ the map
\[
B_\rho(r) \to {\cal B}(H^s, H^{s+ N}), \quad U \mapsto {\cal R}(U)
\]
is continuous and satisfies the \emph{tame} estimate 
\begin{equation}\label{marlene1}
\| {\cal R}(U) \|_{{\cal L}(H^s, H^{s+ N})} \lesssim_{s, N, \rho} \| U \|_\rho^2\,, \quad \forall s \geq \rho\,.  
\end{equation}
\end{definition}
\begin{definition}{\bf (Smoothing operators depending linearly on $(u, \bar u)$).}\label{smoothing lineare u}
Let $N \in \N$. We say that a {\bf bilinear} map 
$(u, w) \mapsto {\cal R}(u)[w]$ belongs to the class ${\cal O}{\cal S}_1( N)$ if it is of the form 
\begin{equation}\label{smoothing bilineare}
{\cal R}(u)[w] = \sum_{\xi, k \in \Z^d} r(k, \xi) \widehat u(k - \xi) \widehat w(\xi) e^{\ii x \cdot k}\,,
\end{equation}
and there exists $\rho \equiv \rho_N > N$ 
such that the {\bf linear} map $H^\rho \to {\cal B}(H^s, H^{s + N})$, 
$u \mapsto {\cal R}(u)$ satisfies the tame estimate 
\begin{equation}\label{marlene2}
\| {\cal R}(u) \|_{{\cal L}(H^s, H^{s + N})} \lesssim_{s, \rho, N} \| u \|_\rho\,, \quad \forall s \geq \rho\,. 
\end{equation}
With a slight abuse of terminology we use the same notation 
for the class of operators of the form 
$(U, w) \mapsto {\cal R}(U)[w] = {\cal R}_+(u) [w]+ {\cal R}_{- }(\overline u)[w]$ 
where ${\cal R}_+, {\cal R}_{-} \in {\cal O}{\cal S}_1( N)$. 
\end{definition}

\begin{definition}{\bf (Smoothing operators).}\label{smoothopera}
We say that ${\cal R}$ is in 
${\cal S}( N)$ if ${\cal R} = {\cal R}_1 + {\cal R}_2$ with ${\cal R} \in {\cal O}{\cal S}_1( N)$ 
and ${\cal R}_2 \in {\cal S}_2(N)$. 
\end{definition}

\begin{definition}[{\bf Matrix valued symbols and operators}]\label{def op simb matrici}
(i)
Consider a matrix valued symbol 
\[
A:=A(U; x, \xi) := \begin{pmatrix}
a(U; x, \xi) & b(U; x, \xi) \\
\overline{b(U; x, - \xi)} & \overline{a(U; x, - \xi)}
\end{pmatrix}
\]
We say that $A \in \Gamma^m_p$, 
resp. $O^m_1$, resp $\Sigma_m^1$ 
if  its entries 
$a, b \in \Gamma^m_p$, resp. $O^m_1$, 
resp. $\Sigma_m^1$. 
We then denote by ${\rm Op}^{bw}(A)$ 
the matrix valued operator 
\begin{equation}\label{barrato4bis}
{\rm Op}^{bw}(A) = 
\begin{pmatrix}
{\rm Op}^{bw}\Big( a(U; x, \xi)  \Big)& {\rm Op}^{bw}\Big( b(U; x, \xi) \Big) \vspace{0.2em}\\
{\rm Op}^{bw}\Big( \overline{b(U; x, - \xi)} \Big)& 
{\rm Op}^{bw}\Big(  \overline{a(U; x, - \xi)} \Big)
\end{pmatrix}
\end{equation}
and we use the same notations to denote the classes given in the definition \ref{Op paradiff}. 

(ii) Similarly if ${\cal R}_1, {\cal R}_2 \in {\cal O}$ where ${\cal O} = {\cal S}_2( N), {\cal O}{\cal S}_1(N), {\cal S}( N)$ we say that 
\begin{equation}\label{barrato4}
{\cal R}(U) = \begin{pmatrix}
{\cal R}_1(U) & {\cal R}_2(U) \\
\overline{{\cal R}_2(U)} & \overline{{\cal R}_1(U)}
\end{pmatrix}
\end{equation}
belongs to the class ${\cal O}$. Here the operators $\ov{R_{j}(U)}$, $j=1,2$, are defined as
\begin{equation}\label{barrato2}
\ov{R_{j}(U)}[h]:=\ov{R(U)[\bar{h}]}\,,\qquad \forall\, h\in H^{s}(\mathbb{T}^{d};\mathbb{C})\,.
\end{equation}
\end{definition}
One can easily check that 
a  linear operator $\mathcal{R} $ of the form \eqref{barrato4} (or \eqref{barrato4bis})
is \emph{real-to-real} in the sense that it 
preserves the spaces ${\bf H}^{s}$ (see \eqref{hcic}).
On the space ${\bf{ H}}^0$ we define the scalar product
\begin{equation}\label{comsca}
(U,V)_{{\bf{ H}}^0}:=
\int_{\mathbb{T}}U\cdot\ov{V}dx.
\end{equation}
Given an operator $\mathcal{R}$  of the form \eqref{barrato4} 
we denote by $\mathcal{R^*}$ its adjoint with respect to the scalar product $\eqref{comsca}$, i.e.
\begin{equation*}
(\mathcal{R}U,V)_{{\bf{ H}}^0}=(U,\mathcal{R}^{*}V)_{{\bf{ H}}^0}\,, 
\quad \forall\,\, U,\, V\in {\bf{ H}}^0.
\end{equation*}
One can check that 
\begin{equation*}
\mathcal{R}^*:=\left(\begin{matrix} \mathcal{R}_1^* & \ov{\mathcal{R}_2}^* 
\\ {\mathcal{R}_2}^* & \ov{\mathcal{R}_1}^*\end{matrix}\right)\,,
\end{equation*}
where $\mathcal{R}_1^*$ and $\mathcal{R}_2^*$ are respectively the adjoints 
of the operators $\mathcal{R}_1$ and $\mathcal{R}_2$ with respect to
the complex scalar product on $L^{2}(\mathbb{T};\mathbb{C})$ defined in \eqref{scalarL}.

\begin{definition}
\label{selfi}
Let $\mathcal{R}$ be an operator as in \eqref{barrato4}.
We say that  $\mathcal{R}$ 
is \emph{self-adjoint} if 
\begin{equation}\label{calu}
\mathcal{R}_1^{*}=\mathcal{R}_1,\;\;
\;\; \ov{\mathcal{R}_2}=\mathcal{R}_2^{*}\,.
\end{equation}
We say that an operator $\mathcal{M}$ as in \eqref{barrato4}
is \emph{Hamiltonian} is $-\ii E\mathcal{M}$ is self-adjoint.
\end{definition}

Consider now a symbol $a=a(x,\x) \in \Gamma^m_p$ ( resp. $O^m_1$, 
resp. $\Sigma_m^1$), and set $A:=\opbw(a(x,\x))$.
Using \eqref{quantiWeyl} and \eqref{barrato2} one can check that 
\begin{align*}
&\bar{A}=\opbw(\widetilde{a}(x,\x))\,,\qquad \widetilde{a}(x,\x)=\ov{a(x,-\x)}\,;
\\
{\bf (Ajdoint)}\qquad 
&A^{*}=\opbw\big(\,\ov{a(x,\x)}\,\big)\,.
\end{align*}
Therefore, a matrix valued paradifferential operator 
as in \eqref{barrato4bis} is self-adjoint according to Definition
\ref{selfi}
if and only if (recall \eqref{calu}) one has
\begin{equation}\label{simboAggiunto2}
a(x,\x)=\ov{a(x,\x)}\,,
\qquad
b(x,-\x)=b(x,\x)\,.
\end{equation}

\begin{definition}[{\bf Symplectic map}]\label{mappasimpl}
Let $\mathcal{Q}=\mathcal{Q}(U)$ be a matrix valued operator of the form \eqref{barrato4} (resp. \eqref{barrato4bis}).
We say that  $\mathcal{Q}$ 
is \emph{symplectic} if 
\begin{equation}\label{condsimpl}
\mathcal{Q}^{*}(-\ii E )\mathcal{Q}=- \ii E\,, \qquad E=\sm{1}{0}{0}{-1}\,.
\end{equation}
\end{definition}

\subsection{Symbolic calculus}\label{sez calcolo simbolico}
In this section we provide some abstract lemmas on the classes 
that we defined before that we shall apply in our normal form procedure. 
We introduce the following differential operator
\[ 
\s(D_{x},D_{\x},D_{y},D_{\eta}) := D_{\x}D_{y}-D_{x}D_{\eta}\,, 
\]
where $D_{x}:=\frac{1}{\ii}\pa_{x}$ and $D_{\x},D_{y},D_{\eta}$ are similarly defined.

\begin{definition}{\bf (Asymptotic expansion of composition symbol).}
Let $\rho\in \mathbb{N}$, $m_1,m_2\in \R$ and $a\in \Sigma_1^{m_1}$, $b\in \Sigma_1^{m_2}$.
We define the symbol
\begin{equation}\label{espansione2}
(a\#_{\rho} b)(U;x,\x):=\sum_{k=0}^{\rho-1}\frac{1}{k!}
\left(
\frac{\ii}{2}\s(D_{x},D_{\x},D_{y},D_{\eta})\right)^{k}
\Big[a(x,\x)b(y,\eta)\Big]_{|_{\substack{x=y, \x=\eta}}}
\end{equation}
modulo symbols in $\Sigma_{1}^{m_1+m_2-\rho \delta}$.
\end{definition}
\begin{remark}\label{espansEsplic}
Recalling \eqref{seminormSimbo} we note that the symbol
$
\s(D_{x},D_{\x},D_{y},D_{\eta})^{k}\big[a(x,\x)b(y,\eta)\big]_{|_{\substack{x=y, \x=\eta}}}
$
belongs to $\Sigma_{1}^{m_1+m_2-\delta k}$. In particular we have the expansion
$a\#_{\rho} b=ab+\tfrac{1}{2\ii}\{a,b\}+\Sigma_{1}^{m_1+m_2-2\delta}$.
\end{remark}
We shall prove the following result on the composition of paradifferential operator.

\begin{proposition}\label{prop:composit}
Fix $\rho\in \mathbb{N}$, $m_1,m_2\in \mathbb{R}$ and $s_0>d/2$. 
There is $q=q(\rho)\gg1$ such that, for $a\in \mathcal{N}^{m_1}_{s_0+q} $, 
$b\in \mathcal{N}^{m_2}_{s_0+q} $,
one has that
\begin{equation}\label{espansionecompo}
\opbw(a)\circ\opbw(b)=\opbw(a\#_{\rho}b)+R(a, b)
\end{equation}
where, for any $s\geq s_0>d/2$, the bilinear and continuous map
$$
{\mathcal N}^{m_1}_{s_0+q(\rho)} \times {\mathcal N}^{m_2}_{s_0+q(\rho)} \to {\cal L}(H^s, H^{s-m_1-m_2+\rho}), \quad (a, b) \mapsto R(a, b)
$$
satisfies
\begin{equation}\label{composit2}
\|R(a, b)h\|_{s-m_1-m_2+\rho}\lesssim_{s}|a|_{m_1,s_0+q(\rho)}
|b|_{m_2,s_0+q(\rho)} \|h\|_{s}\,, \quad \forall h \in H^s\,.  
\end{equation}
\end{proposition}

\begin{proof}
In order to prove the lemma above we reason as follows.
First of all notice that\footnote{We denote the Fourier transform in $x\in \mathbb{T}^{d}$ 
of a function $f(x)$ by $\mathcal{F}(f)(\x)=\widehat{f}(\x)$.}
\begin{equation}\label{def:prodotto1}
\mathcal{F}{(\opbw(a)\circ\opbw(b)h)}(\x)=
\sum_{\eta,\theta\in\mathbb{Z}^{d}}
r_1(\x,\theta,\zeta) \widehat{a}\big(\x-\theta,\frac{\x+\theta}{2}\big)
\widehat{b}\big(\theta-\zeta,\frac{\theta+\zeta}{2}\big)\widehat{h}(\zeta)\,,
\end{equation}
where
\begin{equation}\label{cutoffR1}
r_1(\x,\theta,\zeta):=\eta_{\epsilon}\left(\frac{|\x-\theta|}{|\x+\theta|}\right)
\eta_{\epsilon}\left(\frac{|\theta-\zeta|}{|\theta+\zeta|}\right)\,.
\end{equation}
Fix $L\in \mathbb{N}$ with $L\gg \rho$ to be chosen later.
By Taylor expanding the symbols we have
\begin{equation}\label{expA}
\begin{aligned}
\widehat{a}\Big(\x-\theta,\frac{\x+\theta}{2}\Big)&=
\sum_{k=0}^{L}\frac{1}{2^{k} i^{k}k!}
\widehat{(\pa_{\x}^{k}a)}\Big(\x-\theta,\frac{\x+\zeta}{2}\Big)[\ii(\theta-\zeta)]^{k}+
\\&
+\frac{1}{2^{L+1}\ii^{L+1}L!}\int_{0}^{1}(1-\tau)^{L}
\widehat{(\pa_{\x}^{L+1}a)}\Big(\x-\theta,\frac{\x+\zeta}{2}+\tau\frac{\theta-\zeta}{2}\Big)
[\ii(\theta-\zeta)]^{L+1}d\tau\,,
\end{aligned}
\end{equation}
\begin{equation}\label{expB}
\begin{aligned}
\widehat{b}\Big(\theta-\zeta,\frac{\theta+\zeta}{2}\Big)&=
\sum_{j=0}^{L}\frac{(-1)^{j}}{2^{j} i^{j} j!}
\widehat{(\pa_{\x}^{j}b)}\Big(\theta-\zeta,\frac{\x+\zeta}{2}\Big)[\ii(\x-\theta)]^{j}+
\\&
+\frac{(-1)^{L+1}}{2^{L+1}\ii^{L+1}L!}\int_{0}^{1}(1-\tau)^{L}
\widehat{(\pa_{\x}^{L+1}b)}\Big(\theta-\zeta,\frac{\x+\zeta}{2}+\tau\frac{\theta-\x}{2}\Big)
[\ii(\x-\theta)]^{L+1}d\tau\,.
\end{aligned}
\end{equation}
Therefore we deduce that
\begin{equation}\label{expAB}
\widehat{a}\Big(\x-\theta,\frac{\x+\theta}{2}\Big)
\widehat{b}\Big(\theta-\zeta,\frac{\theta+\zeta}{2}\Big)=
\sum_{\ell=1}^{4}g_{\ell}(\x,\theta,\zeta)
\end{equation}
where
\begin{equation}\label{def:g1}
g_{1}(\x,\theta,\zeta):=
\sum_{p=0}^{L}\frac{1}{2^{p}\ii^{p}p!}
\sum_{k=0}^{p}\left(\begin{matrix}p \\ k\end{matrix}\right)(-1)^{p-k}
\widehat{(\pa_{\x}^{k}\pa_{x}^{p-k}a)}\Big(\x-\theta,\frac{\x+\zeta}{2}\Big)
\widehat{(\pa_{\x}^{p-k}\pa_{x}^{k}b)}\Big(\theta-\zeta,\frac{\x+\zeta}{2}\Big)\,,
\end{equation}
\begin{equation}\label{def:g2}
g_{2}(\x,\theta,\zeta):=
\sum_{p=L+1}^{2L}\frac{1}{2^{p}\ii^{p}p!}
\sum_{k=0}^{p}\left(\begin{matrix}p \\ k\end{matrix}\right)(-1)^{p-k}
\widehat{(\pa_{\x}^{k}\pa_{x}^{p-k}a)}\Big(\x-\theta,\frac{\x+\zeta}{2}\Big)
\widehat{(\pa_{\x}^{p-k}\pa_{x}^{k}b)}\Big(\theta-\zeta,\frac{\x+\zeta}{2}\Big)\,,
\end{equation}
\begin{equation}\label{def:g3}
g_{3}(\x,\theta,\zeta):=\frac{\widehat{b}\Big(\theta-\zeta,\frac{\theta+\zeta}{2}\Big)}{2^{\rho+1}\ii^{\rho+1}\rho!}
\int_{0}^{1}(1-\tau)^{\rho}
\widehat{(\pa_{\x}^{\rho+1}a)}\Big(\x-\theta,\frac{\x+\zeta}{2}+\tau\frac{\theta-\zeta}{2}\Big)
[\ii(\theta-\zeta)]^{\rho+1}d\tau\,,
\end{equation}
\begin{equation}\label{def:g4}
\begin{aligned}
g_{4}(\x,\theta,\zeta)&:=\sum_{k=0}^{\rho}
\frac{(-1)^{\rho+1}}{2^{\rho+1}\ii^{\rho+1}\rho!}\int_{0}^{1}(1-\tau)^{\rho}
\widehat{(\pa_{\x}^{\rho+1}b)}\Big(\theta-\zeta,\frac{\x+\zeta}{2}+\tau\frac{\theta-\x}{2}\Big)
[\ii(\x-\theta)]^{\rho+1}d\tau\times\\
&\times\frac{1}{2^{k} i^{k}k!}
\widehat{(\pa_{\x}^{k}a)}\Big(\x-\theta,\frac{\x+\zeta}{2}\Big)[\ii(\theta-\zeta)]^{k}\,.
\end{aligned}
\end{equation}
We set
\begin{equation}\label{Relle}
\opbw(a)\circ\opbw(b)=\sum_{\ell=1}^{4}R_{\ell}
\end{equation}
where the operators $R_{\ell}$ are defined by 
\begin{equation}\label{def:Relle}
\widehat{R_{\ell}h}(\x)=\sum_{\zeta,\theta\in \mathbb{Z}^{d}}r_{1}(\x,\theta,\zeta)
g_{\ell}(\x,\theta,\zeta)\widehat{h}(\zeta)\,,\qquad \ell=1,\ldots,4\,,
\end{equation}
where $r_1$ is in \eqref{cutoffR1} and $g_{\ell}$ are in \eqref{def:g1}-\eqref{def:g4}.

We now study the explicit form
of the symbol 
$(a\#_{\rho}b)(x,\x)$ (recall \eqref{espansione2}). First of all we note that
(formally)
\[
\frac{1}{p!}\big[\frac{\ii}{2}\s(D_{x},D_{\x},D_{y},D_{\eta})\big]^{p}=
\frac{1}{2^{p}\ii^{p}p!}(\pa_{\x}\pa_{y}-\pa_{x}\pa_{\eta})^{p}=
\frac{1}{2^{p}\ii^{p}p!}\sum_{k=0}^{p}\left(\begin{matrix}p \\ k\end{matrix}\right)(-1)^{p-k}
(\pa_{\x}\pa_{y})^{k}(\pa_{x}\pa_{\eta})^{p-k}\,.
\]
Then it is easy to note that (using \eqref{espansione2} and \eqref{def:g1})
\[
\mathcal{F}(a\#_{\rho}b)(\x-\zeta,\frac{\x+\zeta}{2})=\sum_{\theta\in \mathbb{Z}^{d}}g_{1}(\x,\theta,\zeta)\,.
\]
Hence we have that $ \opbw(a\#_{\rho}b)h=:Q h$ has the form
\begin{equation}\label{RRelle}
\begin{aligned}
 \widehat{Qh}(\x)&:=
\sum_{\zeta\in \mathbb{Z}^{d}} r_{2}(\x,\zeta)
\mathcal{F}(a\#_{\rho}b)(\x-\zeta,\frac{\x+\zeta}{2})\widehat{h}(\zeta)
=\sum_{\zeta,\theta\in \mathbb{Z}^{d}}\chi_{\epsilon}\big(\frac{|\x-\zeta|}{\langle\x+\zeta\rangle}\big)
g_1(\x,\theta,\zeta)\widehat{h}(\zeta)\,,
\end{aligned}
\end{equation}
where
\begin{equation}\label{cutoffRR2}
r_2(\x,\zeta):=\eta_{\epsilon}\big(\frac{|\x-\zeta|}{\langle\x+\zeta\rangle}\big)\,.
\end{equation}
In conclusion, by \eqref{Relle}, \eqref{def:Relle} and \eqref{RRelle}, 
we obtained
\[
\opbw(a)\circ\opbw(b)= \opbw(a\#_{\rho}b)+\mathcal{R}+\sum_{\ell=2}^{4}R_{\ell}
\]
where 
\begin{equation}\label{calRRR}
\begin{aligned}
&\widehat{(\mathcal{R}h)}(\x):=\mathcal{F}\big((R_1-Q)h\big)(\x):=
\sum_{\zeta,\theta\in \mathbb{Z}^{d}} \mathcal{R}(\x,\theta,\zeta)\widehat{h}(\zeta)\\
&\mathcal{R}(\x,\theta,\zeta)\stackrel{\eqref{cutoffR1}}{:=}
\Big[\eta_{\epsilon}\left(\frac{|\x-\theta|}{|\x+\theta|}\right)
\eta_{\epsilon}\left(\frac{|\theta-\zeta|}{|\theta+\zeta|}\right)-
\eta_{\epsilon}\big(\frac{|\x-\zeta|}{\langle\x+\zeta\rangle}\big)
\Big]g_1(\x,\theta,\zeta)
\end{aligned}
\end{equation}
To obtain the \eqref{espansionecompo} it remains to show that the terms
$\mathcal{R}, R_{\ell}$, $\ell=2,3,4$, satisfy the estimate \eqref{composit2}.

We start by considering the remainder $\mathcal{R}$ in \eqref{calRRR}. 
First of all, using the explicit formula \eqref{def:g1} for the coefficients $g_1(\x,\theta,\zeta)$ and reasoning as in Lemma \ref{non hom symbols Fourier},
we deduce that
\begin{equation}\label{decayG1}
|g_1(\x,\theta,\zeta)|\lesssim 
\langle \x-\theta\rangle^{-p}\langle \theta-\zeta\rangle^{-q}|a|_{m_1,p+L}
|b|_{m_2,q+L}
\langle\x+\zeta\rangle^{m_1+m_2}\,,
\end{equation}
for ny $p,q\in \mathbb{N}$.
We now study the properties of the cut-off function $(r_1-r_{2})(\x,\theta,\zeta)$ 
(see \eqref{cutoffR1}, \eqref{cutoffRR2})
appearing in \eqref{calRRR}.
Let us define the sets
\begin{align*}
D&:=\Big\{
(\x,\theta,\zeta)\in \mathbb{Z}^{3d}\; :\;
(r_1-r_2)(\x,\theta,\zeta)=0\Big\}\,,
\\
A&:=\Big\{
(\x,\theta,\zeta)\in \mathbb{Z}^{3d}\; :\;
\frac{|\x-\theta|}{\langle\x+\theta\rangle}\leq \frac{5\epsilon}{4}\,,\;\;
\frac{|\x-\zeta|}{\langle\x+\zeta\rangle}\leq \frac{5\epsilon}{4}\,,\;\;
\frac{|\theta-\zeta|}{\langle\theta+\zeta\rangle}\leq \frac{5\epsilon}{4}\Big\}\,,
\\
B&:=\Big\{
(\x,\theta,\zeta)\in \mathbb{Z}^{3d}\; :\;
\frac{|\x-\theta|}{\langle\x+\theta\rangle}\geq \frac{8\epsilon}{5}\,,\;\;
\frac{|\x-\zeta|}{\langle\x+\zeta\rangle}\geq \frac{8\epsilon}{5}\,,\;\;
\frac{|\theta-\zeta|}{\langle\theta+\zeta\rangle}\geq \frac{8\epsilon}{5}\Big\}\,.
\end{align*}
We note that
\[
D\supseteq A\cup B\quad \Rightarrow\quad D^{c}\subseteq A^{c}\cap B^{c}\,.
\]
Let $(\x,\theta,\zeta)\in D^{c}$ and assume in particular  that 
$(\x,\theta,\zeta)\in{\rm Supp}(r_1):=\ov{\{(\x,\theta,\zeta) : r_1\neq0\}}$. 
Then we can note that
\begin{equation}\label{navyseal2}
|\x-\zeta|\ll \langle\x+\zeta\rangle\,\quad {\rm and}\quad \langle\x\rangle \sim \langle\zeta\rangle.
\end{equation}
Notice also that $(\x,\theta,\zeta)\in{\rm Supp}(r_2)$ implies  the \eqref{navyseal2} as well.
We need to estimate
\[
\|R_0h\|_{{s+\rho-m_1-m_2}}^{2}\lesssim\sum_{\x\in\mathbb{Z}^{d}} \Big(
\sum_{\zeta,\theta}^{*} |g_1(\x,\theta,\zeta)||\hat{h}(\zeta)
|\langle\x\rangle^{s+\rho}
\Big)^{2}=I+II+III\,,
\]
where $\sum_{\zeta,\theta}^{*} $ denotes the sum over indexes satisfying 
\eqref{navyseal2}, the term $I$ denotes the sum on indexes satisfying also
$|\x-\theta|>c\epsilon |\x|$, $II$
denotes the sum on indexes satisfying also
$|\zeta-\theta|>c\epsilon |\zeta|$, for some $0<c\ll1$ and $III$ is defined by difference.
We estimate the term $I$. By using \eqref{navyseal2}, $|\x-\theta|>c\epsilon |\x|$ and \eqref{decayG1}, 
we get
\begin{equation*}
\begin{aligned}
I&
\lesssim
\sum_{\xi\in\Z^d}\Big(\sum_{\zeta,\theta}^{*}|g_1(\x,\theta,\zeta)||\hat{h}(\zeta)| 
\langle \zeta\rangle^s\langle\xi-\theta\rangle^{\rho-m_1-m_2}\Big)^2
\\&
\lesssim |a|^{2}_{m_1,s_0+\rho+L}
|b|^{2}_{m_2,s_0+L}
\||\hat{h}(\xi)| \langle \xi\rangle^s\star\langle\xi\rangle^{s_0+\rho}\star\langle\x\rangle^{-s_0}|\|_{\ell^2(\Z^d)}^2
\\&
\lesssim |a|^{2}_{m_1,s_0+\rho+L}
|b|^{2}_{m_2,s_0+L}
\||\hat{h}(\xi)| \langle \xi\rangle^s\|_{\ell^2(\Z^d)}^2
\lesssim  
|a|^{2}_{m_1,s_0+\rho+L}
|b|^{2}_{m_2,s_0+L}
\|h\|^{2}_{{s}}\,,
\end{aligned}
\end{equation*}
where we used
$s_0>d>d/2$.

Reasoning similarly one obtains  $II\lesssim
\|h\|_{{s}}^{2}|a|^{2}_{m_1+s_0+L}|b|^{2}_{m_2,s_0+\rho+L}$.
The sum $III$
is restricted to indexes satisfying 
\eqref{navyseal2} and $|\x-\theta|\leq c\epsilon |\x|$, $|\zeta-\theta|\leq c\epsilon |\zeta|$.
For $c\ll1$ small enough this restrictions implies that $(\x,\theta,\zeta)\in A$, which 
is a contradiction since $(\x,\theta,\zeta)\in D^{c}\subseteq A^{c}$.

For the remainders $R_{\ell}$, $\ell=2,3,4$ in \eqref{def:Relle} one can reason similarly using 
the explicit formul\ae\, \eqref{def:g2}-\eqref{def:g4}
to show that $g_{\ell}$ are symbols of order at least $L+1$ or $\rho$. Therefore one concludes the proof by choosing 
$L$ large enough.
\end{proof}
By the Proposition above we deduce the following.

\begin{lemma}{\bf (Compositions and commutators).}\label{teoremadicomposizione}

\noindent
$(i)$ Let $a \in \Sigma^m_1$, $b \in \Sigma^{m'}_1$ and let $N \in \N$. 
Then the operator ${\rm Op}^{bw}(a)\circ {\rm Op}^{bw}(b)$ satisfies 
\[
{\rm Op}^{bw}(a)\circ {\rm Op}^{bw}(b) 
= {\rm Op}^{bw}(a b + \frac{1}{2 \ii} \{ a, b \}) + {\rm Op}^{bw}(r_{a b}) + {\cal R}_{a b}(U)\,,
\]
where 
$r_{ab} \in \Gamma_2^{m + m' - 2\delta}$ and the map  
$(U, w) \mapsto {\cal R}_{ab}(U)[w]$ 
belongs to the class ${\cal S}_{2}(N)$.  
As a consequence, the commutator
\[
[{\rm Op}^{bw}(a), {\rm Op}^{bw}(b)] 
=  \frac{1}{\ii}{\rm Op}^{bw}(\{ a, b \}) + {\rm Op}^{bw}(r_{ab} - r_{b a}) 
+ {\cal R}_{a b}(U) - {\cal R}_{b a}(U)\,. 
\]

\noindent
$(ii)$ Let $a \in \Sigma^m_1$ and 
$N \in \N$. 
Then, recalling \eqref{def Lambda xi},  
the Poisson bracket $\{ \Lambda, a \} \in \Sigma_1^{m + 1}$ and 
\[
\begin{aligned}
{\rm Op}^{bw}(\Lambda) \circ{\rm Op}^{bw}(a) 
&= {\rm Op}^{bw}(\Lambda a) + \frac{1}{2 \ii} {\rm Op}^{bw}(\{ \Lambda, a \})
+ {\rm Op}^{bw}(r_{\Lambda a}) 
+\mathcal{R}_{\Lambda a}(U)\,,\\
[{\rm Op}^{bw}(\Lambda), {\rm Op}^{bw}(a)] 
&= \frac{1}{\ii}{\rm Op}^{bw}(\{ \Lambda, a \}) 
+ {\rm Op}^{bw}(r_{\Lambda a}-r_{a \Lambda}) 
+\mathcal{R}_{\Lambda a}(U)-\mathcal{R}_{a \Lambda }(U)\,,
\end{aligned}
\]
where $r_{\Lambda a} \in \Sigma_1^{m +1- 2\delta}$ and   
${\cal R}_{\Lambda a}(U), {\cal R}_{a \Lambda}(U)$ are in ${\cal S}(N)$.   
\end{lemma}

\begin{proof}
It follows by Proposition \ref{prop:composit}, using formula \eqref{espansione2}. The homogeneity 
expansions of symbols and remainders can be deduced 
by the formul\ae\, in the proof of the Proposition.
\end{proof}

We also have the following result about the composition between the smoothing operators introduced
in Def. \ref{nonomosmooth}-\ref{smoothopera}.

\begin{lemma}
Let $N\in \mathbb{N}$, $m\in \mathbb{R}$, $a\in \Sigma_1^{m}$ and $R,Q\in \mathcal{S}(N)$.
Then one has 

\noindent
$(i)$
$ R(U)\circ Q(U) $ and $ Q(U)\circ R(U) $ 
are smoothing operators in $\mathcal{S}_2(N)$.

\noindent
$(ii)$ 
$ R(U)\circ \opbw(a(U;x,\x)) $, $  \opbw(a(U;x,\x))\circ R(U) $
are   in $\mathcal{S}_2(N-m)$.
\end{lemma}

\begin{proof}
By Definition \ref{smoothopera} we can write $R=R_1+R_2$, $Q=Q_1+Q_2$
for some  $R_1,Q_1\in \mathcal{OS}_1(N)$ and $R_2,Q_2\in \mathcal{S}_2(N)$
(see Def. \ref{nonomosmooth}-\ref{smoothing lineare u}).
Then item $(i)$ follows by using estimates \eqref{marlene1} and \eqref{marlene2}.
Item $(ii)$ follows similarly by using also Lemma \ref{standard cont bony}
and Remark \ref{stime simboli in omogeneita}.
\end{proof}

\section{Technical Lemmata}

\subsection{Flows and conjugations}
In this section we prove some abstract results about the conjugation of 
paradifferential operators and smoothing remainders under flows.

Consider a real symbol $g \in \Sigma^m_1$ with $m < 1$ and 
 the flow $\Phi^\tau_g(U)$, $\tau\in[-1,1]$
 defined by 
 \begin{equation}\label{flusso1para}
 \left\{\begin{aligned}
 &\pa_{\tau}\Phi^{\tau}_{g}(U)=\ii G(U)\Phi^{\tau}_{g}(U)\,,\qquad G(U):=\opbw(g(U;x,\x))\,,\\
 &\Phi^{0}(U)=\uno\,.
 \end{aligned}\right.
 \end{equation}
 We have the following.
 \begin{lemma}{\bf (Linear flows).}\label{lemm:flussoGGG}
 There are $s_0>d/2$ and $r>0$ such that, for any $U=\vect{u}{\bar{u}}$ with
 $u\in H^{s}(\mathbb{T}^{d};\mathbb{C})\cap B_{s_0}(r)$, for any $s>0$
 the problem \eqref{flusso1para} admits a unique solution $\Phi^{\tau}_{g}(U)$
 satisfying
 
 \begin{equation}\label{stime flusso elementary}
 \begin{aligned}
\| \Phi_g^\tau (U) w \|_{s} &\leq \| w \|_s (1 + C(s) \| u \|_\rho)\,, 
\qquad \forall\, w\in H^{s}(\mathbb{T}^{d};\mathbb{C})\,,\\
 \| (\Phi_g^\tau(U) - {\rm Id}) v \|_s &\lesssim_s \| u \|_\rho \| v \|_{s + m}\,, 
 \qquad\forall\, v\in H^{s+m}(\mathbb{T}^{d};\mathbb{C})\,,
 \end{aligned}
\end{equation}
for some $C(s)>0$, 
uniformly in $\tau\in[0,1]$.
The map (see \eqref{hcic})
\begin{equation}\label{Phicic}
{\bf \Phi}^{\tau}_{g}(U):=\left(
\begin{matrix}
\Phi_g^\tau (U) \vspace{0.2em}\\
\ov{\Phi_g^\tau (U)}
\end{matrix}
\right)\; : \; {\bf H}^{s}\to {\bf H}^{s}
\end{equation}
is symplectic according to Definition \ref{mappasimpl}.
 \end{lemma}
 \begin{proof}
 The result follows by a standard energy estimate
 using the fact that the symbol $g(U;x,\x)$ is \emph{real} valued.  For more details
 we refer to Lemma 3.22 in \cite{BD}.
 The map ${\bf \Phi}_{g}^{\tau}$ in \eqref{Phicic} can be seen as the linear flow generated by
 the field
 $\mathcal{G}(U)=\ii E\uno G(U)$. Therefore one can check that it is symplectic 
 by reasoning as in Lemma 2.1 in \cite{Feola-Iandoli-Loc}.
 \end{proof}

We set $\Phi_g(U) := \Phi^1_g(U)$ and its inverse $\Phi_g(U)^{-1} : = \Phi^\tau_g(U)_{| \tau = - 1}$. 
The following lemma holds. 

\begin{lemma}{\bf (Conjugation of operators under paradifferential flows)}\label{coniugio flusso diag}
Let $g \in \Sigma^n_1$ with $n < \delta$ and 
assume that $g(U; x, \xi)$ is a real symbol. 
Then the following holds. 

\noindent
$(i)$ If $a \in \Sigma^m_1$, for any fixed $N \in \N$, one has 
\[
\Phi_g(U)^{- 1} {\rm Op}^{bw}(a) \Phi_g(U) = 
{\rm Op}^{bw}(a) + {\rm Op}^{bw}(b) + {\cal R}(U)\,,
\]
where $b \in \Gamma_2^{m + n - \delta}$, ${\cal R} \in {\cal S}_2( N)$.
If the symbol $a$ is real valued, then $b$ is real valued as well.

\noindent
$(ii)$ For any fixed $N \in \N$, one has  (see \eqref{def Lambda xi})
\[
\Phi_g(U)^{- 1} {\rm Op}^{bw}(\Lambda) \Phi_g(U) 
= {\rm Op}^{bw}(\Lambda) + {\rm Op}^{bw}(\{ \Lambda, g \}) 
+ {\rm Op}^{bw}(b) + {\cal R}(U) \,,
\]
where $b$ is a real valued symbol in $\Sigma_1^{n + 1 -(\delta - n)}$ and ${\cal R} \in {\cal S}(N)$. 

\noindent
$(iii)$ Let ${\cal R}$ be in ${\cal S}( N)$. 
Then ${\cal R}_1(U) := \Phi_g(U)^{- 1} {\cal R}(U) \Phi_g(U)$ 
is in the class ${\cal S}( N - n)$. 
\end{lemma}

\begin{proof}
\emph{Item} $(i)$.
Using \eqref{flusso1para} we get, for $L\geq 3$, the 
 Lie  expansion 
\begin{align}
 \Phi{g}(U)^{-1} \opbw(a)&\Phi_{g}(U)=\opbw(a)+ \big[ \opbw(a),  \opbw(\ii g) \big]
+ \sum_{k=2}^{L}\frac{(-1)^{k}}{k!}{\rm Ad}^{k}_{\opbw(\ii g)}[\opbw(a)]\nonumber
\\&
+\frac{(-1)^{L+1}}{L!}\int_{0}^{1} (1-\theta)^{L} \Phi^{-\theta}(U)
\big({\rm Ad}^{L+1}_{\opbw(\ii g)}[\opbw(a)] \big)\Phi^{\theta}(U) d \theta\,,\label{int-Lie}
\end{align}
where we defined ${\rm Ad}_{G}[A]:=[G,A]$ and 
${\rm Ad}^{k}_{G}[A]:={\rm Ad}_{G}\big[{\rm Ad}^{k-1}_{G}[A]\big]$
for $k\geq2$.
By Lemma \ref{teoremadicomposizione} 
(possibly replacing the smoothing index $N$ 
by some $\tilde{N}$  chosen below large enough)
and Remark \ref{espansEsplic} 
we get
$$
{\rm Ad}_{\opbw(\ii g)}[\opbw(a)]  =\big[\opbw(\ii g), \opbw(a)\big] 
 = \opbw\big( \{  g, a \} + r_1 \big) \, , \quad  r_1\in\Sigma_{1}^{m+n-2\delta} \, , 
$$
up to a smoothing operator in $\mathcal{S}_2(\tilde{N}-m-n)$.
Similarly, 
by  induction, for $ k \geq 2 $ we have  
\[
{\rm Ad}^{k}_{\opbw(\ii g)}[\opbw(a)]=\opbw( b_k),\quad 
b_k\in
\Sigma_1^{k(n-\delta)+m} \, , 
\]
up to a smoothing operator in  $\mathcal{S}_2(\tilde{N}-m-kn)$.
We choose $L$  in such a way that  
$  (L+1)(\delta-n) - m \geq \rho$ and $ L + 1 \geq 3 $,  so that the operator 
$ \opbw( b_{L+1}) $ belongs to $\mathcal{S}_2(N)$.
The integral Taylor remainder in \eqref{int-Lie} 
belongs to $ \mathcal{S}_2(N)$
 as well by item $(iii)$ that we proved above.
Then we choose $\tilde{N}$  large enough so that $ \tilde{N} - m - ( L + 1)n \geq  N $
and the remainders are  $N$-smoothing.
Assume now that $a\in \Sigma_1^{m}$ is real valued. Using formula \eqref{espansione2} 
one can check that also the symbol $b$ constructed through the expansion above is real valued.

Item $(ii)$ follows by reasoning as done for item by replacing $a$ with the symbol 
 $\Lambda(\xi) := \| \xi \|_g^2 + m$. 
 Item $(iii)$ follows by using estimates \eqref{marlene1}, \eqref{marlene2} 
 on the remainder $\mathcal{R}$
and the second estimate in \eqref{stime flusso elementary} 
on the map $\Phi_{g}(U)$.
 This concludes the proof.
\end{proof}

Consider now a smooth vector field $X_{NLS}: {\bf H}^{s}\to {\bf H}^{s-2}$ (see \eqref{hcic})
satisfying, for $s\gg1$, 
\begin{equation}\label{ipovector}
\begin{aligned}
\|X_{NLS}(U)\|_{{s-2}}
&\lesssim_{s}\|u\|_{{s}}(1+\|u\|_{{s}})\,,\qquad \forall\, U=\vect{u}{\bar{u}}\in {\bf H}^{s}\,,
\\
\|dX_{NLS}(U)[H_1]\|_{{s-2}}&\lesssim_{s}\|H_1\|_{{s}}(1+\|u\|_{{s}})\,,\qquad 
\forall\, U,H_1\in{\bf H}^{s}\,,
\\
\|d^n X_{NLS}(U)[H_1, \ldots, H_n]\|_{{s-2}}
&\lesssim_{s}  \| H_1 \|_{{s}} \ldots \| H_n \|_{{s}}\,,
\qquad \forall\,, U,H_1,\ldots,H_{n}\in{\bf H}^{s}\,,\;\;n\geq2\,.
\end{aligned}
\end{equation}

\begin{lemma}\label{paTsimbopsi}
Let $g \in \Sigma^m_1$ and 
assume that $U(t, x)$ is a solution belonging to $C^{0}([0,T];{\bf H}^{s})$, $T>0$, $s\gg1$
of the 
Schr\"odinger equation $\partial_t U = X_{NLS}(U)$. 
Then $\partial_t \psi(U(t); x, \xi) = a_\psi(U(t); x, \xi)$ where 
the symbol $a_\psi(U; x, \xi)$ belongs to the class $\Sigma_1^m$
with estimates uniform in $t\in[0,T]$.
\end{lemma}

\begin{proof}
One has that 
\[
\partial_t \psi(U(t); x, \xi) = d \psi(U(t); x, \xi)[\partial_t U] = d \psi(U(t) ; x, \xi)[X_{NLS}(U(t))]\,.
\]
Hence the symbol $a_\psi$ is defined by 
$
a_\psi(U; x, \xi) := d \psi(U ; x, \xi)[X_{NLS}(U)]\,.
$
Then the result follows by using Remark \ref{stime simboli in omogeneita} and 
estimates  \eqref{ipovector}.
\end{proof}

\begin{lemma}{\bf (Conjugation of $\pa_{t}$ under paradifferential flows).}\label{coniugio partial t flusso diag}
Let $g \in \Sigma_1^n$ with $n < \delta$ and $g(U; x, \xi)$. 
Consider a vector field $X_{NLS}$ satisfying \eqref{ipovector}.
Assume that $\partial_t U(t) = X_{NLS}(U(t))$ and $U\in C^{0}([0,T];{\bf H}^{s})$ for some $T>0$, $s\gg1$. 
Then for any $N \in \N$
\[
\Phi_g(U(t))^{- 1} \circ \partial_t \circ \Phi_g(U(t)) = 
\partial_t + {\rm Op}^{bw}(b(U(t); x, \xi))) + {\cal R}(U(t))\,,
\]
where $b(U; x, \xi)$ is a purely imaginary symbol in  $ \Sigma^{n}_1$ 
and the map $(U, w) \mapsto {\cal R}(U)[w]$ is in the class ${\cal S}(N)$.  
\end{lemma}

\begin{proof}
Fix $L\geq 3$. By classical Lie expansions we obtain
\begin{align*}
 \Phi{g}(U)^{-1} \pa_{t}&\Phi_{g}(U)=\pa_{t}+  \opbw(\ii \pa_{t}g) 
+ \sum_{k=2}^{L}\frac{(-1)^{k-1}}{k!}{\rm Ad}^{k-1}_{\opbw(\ii g)}[\opbw(\ii\pa_{t}g)]\nonumber
\\&
+\frac{(-1)^{L}}{L!}\int_{0}^{1} (1-\theta)^{L} \Phi^{-\theta}(U)
\big({\rm Ad}^{L}_{\opbw(\ii g)}[\opbw(\ii \pa_{t}g)] \big)\Phi^{\theta}(U) d \theta\,,
\end{align*}
where we used that $[\pa_{t}, \opbw(\ii g)]=\opbw(\ii \pa_{t}g)$.
By Lemma \ref{paTsimbopsi} we have that $\pa_{t}g$ is a symbol in $\Sigma_{1}^{n}$ with estimates uniform in 
$t\in [0,T]$. The one concludes arguing as done in Lemma \ref{coniugio flusso diag}.
\end{proof}

\vspace{0.5em}
\noindent
In our procedure, we also need to consider the maps of the form 
$\Phi_{\psi}(U):=\Phi_{\psi}^{1}(U)$,
$\Phi_{\mathcal{F}}(U):=\Phi_{\mathcal{F}}^{1}(U)$
where
$\Phi_{\psi}^{\tau}(U)$, $\Phi_{\mathcal{F}}^{\tau}(U)$, $\tau\in[0,1]$ are given by
 \begin{equation}\label{off diag astratto}
\pa_{\tau}\Phi^{\tau}_{\psi}(U)=\ii\opbw\begin{pmatrix}
0 &\psi(U; x, \xi) \\
-\overline{\psi(U; x, -\xi)} & 0
\end{pmatrix} \Phi^{\tau}_{\psi}(U)\,,\quad \Phi^{0}_{\psi}(U)=\uno\,,\quad
\psi \in \Sigma_1^{- n}, \quad n \in \N
 \end{equation}
 \begin{equation}\label{flusso smoothing}
\partial_\tau \Phi^\tau_{\cal F}(U) = {\cal F}(U) \Phi^\tau_{\cal F}(U), \quad \Phi^0_{\cal F}(U) = \uno\,,
\qquad \mathcal{F}(U)\in {\cal O}{\cal S}_1(N)\,.
\end{equation}
We only state the conjugacy properties with the flows $\Phi^\tau_\psi(U)$ and $\Phi_{\cal F}(U)$. The proofs can be done arguing as in Lemmata \ref{lemm:flussoGGG}-\ref{coniugio partial t flusso diag}, with the obvious modifications.

\begin{lemma}\label{well-posflusso2}
 There are $s_0>d/2$ and $r>0$ such that, for any $U=\vect{u}{\bar{u}}$ with
 $u\in H^{s}(\mathbb{T}^{d};\mathbb{C})\cap B_{s_0}(r)$, for any $s>s_0$
 the problems \eqref{off diag astratto} and \eqref{flusso smoothing} admit unique solutions $\Phi^{\tau}_{\psi}(U)$,
 $\Phi^{\tau}_{\mathcal{F}}(U)$
 satisfying
\begin{equation*}
\begin{aligned}
&\| \Phi^\tau_{\psi}(U) \|_{{\cal L}(H^s)} \leq 1 + C(s) \| u \|_{s_0}\,, 
\quad \| \Phi^\tau_{\psi}(U) - {\rm Id} \|_{{\cal L}(H^s, H^{s + n})} \lesssim_s \| u \|_{s_0}
\\&\| \Phi^\tau_{\cal F}(U) \|_{{\cal L}(H^s)} \leq 1 + C(s) \| u \|_{s_0}\,, 
\quad \| \Phi^\tau_{\cal F}(U) - {\rm Id} \|_{{\cal L}(H^s, H^{s + N})} \lesssim_s \| u \|_{s_0}\,, \quad \forall s \geq s_0\,,
\end{aligned}
\end{equation*}
uniformly in $\tau\in[-1,1]$.
Moreover the maps 
$\Phi^{\tau}_{\psi}(U)$,
 $\Phi^{\tau}_{\mathcal{F}}(U)$
are symplectic according to Definition \ref{mappasimpl}.
\end{lemma}
\begin{proof}
It follows by standard theory of ODEs in Banach space.
\end{proof}

We set $\Phi_{\psi}(U) :=\Phi_{\psi}^1(U)$
and $\Phi_{\cal F}(U) :=\Phi_{\cal F}^1(U)$.  Their inverse are  given by 
$\Phi_{\psi}(U)^{- 1} = \Phi^\tau_{\psi}(U)_{| \tau = - 1}$
and $\Phi_{\cal F}(U)^{- 1} = \Phi^\tau_{\cal F}(U)_{| \tau = - 1}$. 
The following Lemmata can be deduce 
by reasoning exactly as done in Lemmata \ref{coniugio flusso diag}, \ref{paTsimbopsi}
and \ref{coniugio partial t flusso diag}. Hence we omit their proofs.

\begin{lemma}\label{proprieta coniugio off diagonal}

\noindent
$(i)$ Let $A$ be a matrix valued symbol in $\Sigma^m_1$ 
as in the definition \ref{def op simb matrici} and 
let $\Phi_{\psi}(U)$ as in \eqref{off diag astratto}. 
Then for any $N \in \N$, 
\[
\Phi_{\psi}(U)^{- 1} {\rm Op}^{bw}(A) \Phi_{\psi}(U) = {\rm Op}^{bw}(A) + {\rm Op}^{bw}(B) + {\cal R}(U)
\]
where $B \in \Gamma_2^{m - n}$ and the ${\cal R}(U)$ belongs to the class ${\cal S}_2( N)$. 
If the matrix of symbols $A$ satisfies the conditions \eqref{simboAggiunto2}, then 
the matrix $B$ satisfies \eqref{simboAggiunto2} as well.

\noindent
$(ii)$ One has that
\[
\begin{aligned}
\Phi_{\psi}(U)^{- 1} \ii E {\rm Op}^{bw}(\Lambda) \Phi_{\psi}(U)&  =  
\ii E {\rm Op}^{bw}(\Lambda) + \opbw\begin{pmatrix}
0 &- 2 \Lambda(\xi) \psi(U; x, \xi)
\\
2 \Lambda(- \xi) \overline{\psi(U; x, - \xi)} & 0
\end{pmatrix}  
\\& 
\quad + {\rm Op}^{bw}(B) + {\cal R}(U)
\end{aligned}
\]
where $B$ is a matrix  in $\Sigma_1^{1 - n}$ satisfying \eqref{simboAggiunto2}
and ${\cal R}(U)$ is in the class ${\cal S}(N)$. 

\noindent
$(iii)$ Assume that $U\in C^{0}([0,T];{\bf H}^{s})$ for some $T>0$, $s\gg1$,
solves 
 $\partial_t U(t) = X_{NLS}(U(t))$ where $X_{NLS}$ satisfies \eqref{ipovector}. 
 Then for any $N \in \N$
\[
\Phi_{\psi}(U(t))^{- 1} \circ \partial_t \circ \Phi_{\psi}(U(t)) = 
\partial_t + {\rm Op}^{bw}(B(U(t); x, \xi))) + {\cal R}(U(t))
\]
where $B(U; x, \xi) \in \Sigma^{- n}_1$ and the map 
$(U, w) \mapsto {\cal R}(U)[w]$ is in the class ${\cal S}(N)$. 
Moreover the matrix of symbols $-\ii EB$ satisfies the conditions \eqref{simboAggiunto2}.

\noindent
$(iv)$ Let ${\cal R}(U)$ be in the class ${\cal S}(N)$. 
Then $\Phi_{\psi}(U)^{- 1} {\cal R}(U) \Phi_{\psi}(U)$ is in the class ${\cal S}(N)$. 
\end{lemma}


\begin{lemma}\label{coniugazioni flusso smoothing}
Let $N \in \N$, ${\cal F} \in {\cal O}{\cal S}_1(N)$. 
Then the following holds

\noindent
$(i)$ Let $A \in \Sigma_1^m$ be a matrix valued symbol. 
Then for any $\tau \in [- 1 , 1]$, 
\[
\Phi_{\cal F}(U)^{- 1} {\rm Op}^{bw}(A) \Phi_{\cal F}(U) = {\rm Op}^{bw}(A) + {\cal R}(U)
\]
where ${\cal R}(U) \in {\cal S}_2(N - m)$. 

\noindent
$(ii)$  One has that
\[
\Phi_{\cal F}(U)^{- 1} \ii E {\rm Op}^{bw}(\Lambda) \Phi_{\cal F}(U) 
= \ii E {\rm Op}^{bw}(\Lambda) + [\ii E {\rm Op}^{bw}(\Lambda), {\cal F}(U)] + {\cal R}(U)
\]
for some ${\cal R}(U)$ in the class ${\cal S}_2(N - 2)$. 

\noindent
$(iii)$ Assume that $U \in C^0([0, T], H^{\rho + 2}) \cap C^1([0, T], H^\rho)$ 
solves the equation $\partial_t U = X_{NLS}(U)$ 
for some $\rho \equiv \rho_N \geq N$ large enough
where $X_{NLS}$ satisfies \eqref{ipovector}.
Then 
\[
\Phi_{\cal F}(U)^{- 1} \circ \partial_t  \circ \Phi_{\cal F}(U) = 
\partial_t - {\cal F}(\ii E {\rm Op}^{bw}(\Lambda) U(t)) + {\cal R}(U(t))
\]
where ${\cal R}(U)$ belongs to the class ${\cal S}_2(N)$. 

\noindent
$(iv)$ Let $N' \in \N$, ${\cal R}(U)$ be in the class ${\cal S}(N')$. 
Then 
\[
\Phi_{\cal F}(U)^{- 1} {\cal R}(U) \Phi_{\cal F}(U) = {\cal R}(U) + {\cal Q}(U)
\]
where ${\cal Q}(U)$ is in the class ${\cal S}_2(N + N')$.
\end{lemma}


\subsection{Some calculus about smoothing operators}
In this section we proof some abstract results on linear smoothing operators
introduced in Definition \ref{smoothing lineare u}. These results will be used
in Section \ref{sec:BNFstep} and they are based on the estimates on the small divisors
proved in Appendix \ref{sezione non rionanza}.

\begin{lemma}\label{stime eq omologica birkhoff}
Let $\mathcal{G}\in (0,+\infty)$ be the full Lebesgue measure set
given by Lemma \ref{condizioni di non-risonanza}. 
Then for any $m\in \mathcal{G}$ 
the following holds.
Let 
\[
{\cal R}(u) w = \sum_{k, \xi \in \Z^d} r(k, \xi) \widehat u(k - \xi) \widehat w(\xi) e^{\ii k \cdot x}
\] 
be in the class ${\cal O}{\cal S}_1(N)$. 
Define for any $\sigma, \sigma' \in \{ + 1, - 1\}$, the operator 
${\cal F}_{\sigma, \sigma'}(u)$ as 
\begin{equation}\label{def cal F sigma sigma prime}
{\cal F}_{\sigma, \sigma'}(u)[w] := 
- \sum_{k, \xi \in \Z^d} \dfrac{r(k, \xi)}{\ii \Big(\Lambda(k) + \sigma \Lambda(k - \xi) + \sigma' \Lambda(\xi) \Big)} 
\widehat u(k - \xi) \widehat w(\xi) e^{\ii k \cdot x}\,.
\end{equation}
Then ${\cal F}_{\sigma, \sigma'}(u)$ is in the class ${\cal O}{\cal S}_1(N - \tau)$ 
and solves the equation 
\begin{equation}\label{omologica pastiera 101}
{\rm Op}^{bw}(\Lambda) {\cal F}_{\sigma, \sigma'}(u) 
+ \sigma {\cal F}_{\sigma, \sigma'}({\rm Op}^{bw}(\Lambda) u ) 
+ \sigma' {\cal F}_{\sigma, \sigma'}(u) {\rm Op}^{bw}(\Lambda) + {\cal R}(u) = 0\,.
\end{equation}
\end{lemma}

\begin{proof}
We prove the claimed statement in the case where $\sigma = \sigma' = - 1$. 
The other cases can be proved similarly. 
It is immediate to verify that ${\cal F}_{\sigma, \sigma'}$ 
defined in \eqref{def cal F sigma sigma prime} solves the equation \eqref{omologica pastiera 101}. 
Since ${\cal R}(u)$ is in the class ${\cal O}{\cal S}_1(N)$, 
one has that there exists $\rho \equiv \rho_N > N$ 
large enough such that for any $s \geq \rho$, 
\[
 \sum_{k \in \Z^d} \langle k \rangle^{2(s + N)} \Big| \sum_{\xi \in \Z^d} 
 r(k, \xi) \widehat u(k - \xi) \widehat w(\xi) \Big|^2 
 = \| {\cal R}(u) w \|_{s + N}^2 \lesssim_s \| u \|_\rho^2 \| w \|_s^2\,.
\]
By taking $w (x) = e^{\ii x \cdot \xi}$, 
the latter estimate implies 
\begin{equation}\label{pastiera al cioccolato 100}
\sum_{k \in \Z^d} \langle k \rangle^{2 (s + N)} |r(k, \xi)|^2 |\widehat u(k - \xi)|^2  
\lesssim_s \| u \|_\rho^2 \langle \xi \rangle^{2 s}\,, 
\quad \forall s \geq \rho\,, \quad \forall u \in H^\rho\,.  
\end{equation}
Let ${\cal F}(u) := {\cal F}_{- 1, - 1}(u)$ (see \eqref{def cal F sigma sigma prime}). 
One has that 
\[
\begin{aligned}
 {\cal F}(u)[w] &= \sum_{k, \xi \in \Z^d} f(k, \xi) \widehat u(k - \xi) \widehat w(\xi) e^{\ii k \cdot \xi}\,, 
\\
f(k, \xi) &:= - \dfrac{r(k, \xi)}{\ii \Big( \Lambda(k) - \Lambda(k - \xi) - \Lambda(\xi) \Big)}\,, 
\qquad k, \xi \in \Z^d\,. 
\end{aligned}
\]
By the bounds \eqref{lowerstima} given 
by Lemma \ref{condizioni di non-risonanza},
one has that there exists $\tau = \tau(d) \gg 0$, large enough and 
$\gamma \in (0, 1)$ small enough such that 
\[
|\Lambda(k) - \Lambda(k - \xi) - \Lambda(\xi) | 
\geq \frac{\gamma}{\langle k - \xi \rangle^\tau \langle \xi \rangle^\tau}, \quad \forall k, \xi \in \Z^d\,,
\]
and therefore 
\begin{equation}\label{f k xi r k xi}
|f(k, \xi)| \lesssim \langle k - \xi \rangle^\tau \langle \xi \rangle^\tau |r(k, \xi)|\,. 
\end{equation}
We now estimate the norm $\| {\cal F}(u) w \|_{s - \tau + N }$. 
Take $s - \tau \geq \rho$ in such a way that \eqref{pastiera al cioccolato 100} 
holds with $s - \tau$ instead of $s$. 
Using the Cauchy-Schwartz inequality 
(using that $\sum_\xi \langle k - \xi \rangle^{- 2 s_0} \leq C < \infty$ for $s_0 > d/2$), 
one has 
\[
\begin{aligned}
\| {\cal F}(u) w \|_{s - \tau + N }^2 
& \leq 
\sum_{k \in \Z^d} \langle k \rangle^{2(s - \tau + N )} 
\Big( \sum_{\xi \in \Z^d} |f(k, \xi)| |\widehat u(k - \xi)| |\widehat w(\xi)|  \Big)^2  
\\& 
\stackrel{\eqref{f k xi r k xi}}{\lesssim} 
\sum_{k \in \Z^d} \langle k \rangle^{2(s - \tau  + N )} 
\Big( \sum_{\xi \in \Z^d} |r(k, \xi)| \langle k - \xi \rangle^\tau 
|\widehat u(k - \xi)| \langle \xi \rangle^\tau |\widehat w(\xi)|  \Big)^2 
\\& 
\lesssim \sum_{k, \xi \in \Z^d} \langle k \rangle^{2(s - \tau + N )}   
|r(k, \xi)|^2 \langle k - \xi \rangle^{2 (\tau + s_0)} |\widehat u(k - \xi)|^2 
\langle \xi \rangle^{2 \tau} |\widehat w(\xi)|^2   
\\& 
= \sum_{k , \xi \in \Z^d} \langle k \rangle^{2(s - \tau + N) }  |r(k, \xi)|^2  
|\widehat{ \langle D \rangle^{\tau + s_0} u}(k - \xi)|^2  
|\widehat{\langle D \rangle^\tau w}(\xi)|^2  
\\& 
=\sum_{\xi \in \Z^d} |\widehat{\langle D \rangle^\tau w}(\xi)|^2 
\Big( \sum_{k \in \Z^d} \langle k \rangle^{2(s - \tau + N) }  |r(k, \xi)|^2  
|\widehat{ \langle D \rangle^{\tau + s_0} u}(k - \xi)|^2 \Big) 
\\& 
\stackrel{\eqref{pastiera al cioccolato 100}}{\lesssim_s} 
\| \langle D \rangle^{\tau + s_0} u \|_\rho^2 \sum_{\xi \in \Z^d}
\langle \xi \rangle^{2 (s - \tau)} |\widehat{\langle D \rangle^\tau w}(\xi)|^2 
\lesssim_s \| u \|_{\rho + \tau + s_0}^2 \| w \|_{s}^2\,.
\end{aligned}
\]
The latter chain of inequalities implies that for any $s \geq \rho' := \rho + \tau + s_0$, the map 
\[
H^{\rho'} \to {\cal L}(H^s, H^{s + N - \tau}), \quad u \mapsto {\cal F}(u)
\]
is a bounded linear map. 
This implies that ${\cal F}$ belongs to the class ${\cal O}{\cal S}_1(N - \tau)$. 
\end{proof}
As a consequence of the latter lemma, one gets the following.

\begin{lemma}\label{eq omologica Birkhoff step}
Let $\mathcal{G}\in (0,+\infty)$ be the full Lebesgue measure set
given by Lemma \ref{condizioni di non-risonanza}. 
Then for any $m\in \mathcal{G}$ 
the following holds.
Let ${\cal R} \in {\cal O}{\cal S}_1(N)$ be a matrix valued operator. 
Then there exists a matrix valued operator ${\cal F} \in {\cal O}{\cal S}_1(N - \tau)$ 
which solves the equation 
\begin{equation}\label{lemma astratto omologica birkhoff}
- {\cal F}(\ii E {\rm Op}^{bw}(\Lambda) U) + [\ii E {\rm Op}^{bw}(\Lambda), {\cal F}(U)] 
+ {\cal R}(U) = 0   \,. 
\end{equation}
\end{lemma}

\begin{proof}
The operator ${\cal R} \in {\cal O}{\cal S}_1(N)$ has the form 
\begin{equation*}
{\cal R}(U) = 
\begin{pmatrix}
{\cal R}_1(U) & {\cal R}_2(U) \vspace{0.2em}\\
\overline{{\cal R}_2(U)} & \overline{{\cal R}_1(U)}
\end{pmatrix} = 
\begin{pmatrix}
{\cal R}_1^+(u) + {\cal R}_1^-(\overline u) & {\cal R}_2^+(u) + {\cal R}_2^-(\overline u) \vspace{0.2em}\\
\overline{{\cal R}_2^+(u)} + \overline{{\cal R}_2^-(\bar u)} & \overline{{\cal R}_1^+(u)} 
+ \overline{{\cal R}_1^-(\overline u)}
\end{pmatrix}\,.
\end{equation*}
One looks for ${\cal F} \in {\cal O}{\cal S}_1(N - \tau)$ of the same form, namely 
\[
{\cal F}(U) = \begin{pmatrix}
{\cal F}_1^+(u) + {\cal F}_1^-(\overline u) & {\cal F}_2^+(u) + {\cal F}_2^-(\overline u) \vspace{0.2em} \\
\overline{{\cal F}_2^+(u)} 
+ \overline{{\cal F}_2^-(\bar u)} & \overline{{\cal F}_1^+(u)} 
+ \overline{{\cal F}_1^-(\overline u)}
\end{pmatrix}\,.
\]
A direct calculation shows that the equation 
\eqref{lemma astratto omologica birkhoff} is equivalent to 
\[
\begin{aligned}
& \ii \Big( - {\cal F}_1^+({\rm Op}^{bw}(\Lambda) u) 
+  [{\rm Op}^{bw}(\Lambda), {\cal F}_1^+(u)] \Big)  + {\cal R}_1^+(u) = 0,
\\& 
\ii \Big( {\cal F}_1^-({\rm Op}^{bw}(\Lambda) \overline u) 
+ [{\rm Op}^{bw}(\Lambda), {\cal F}_1^-(\overline u)] \Big)+ {\cal R}_1^-(\overline u) = 0\,, 
\\& 
\ii \Big( - {\cal F}_2^+({\rm Op}^{bw}(\Lambda) u) 
+  {\rm Op}^{bw}(\Lambda) {\cal F}_2^+(u) + {\cal F}_2^+(u) {\rm Op}^{bw}(\Lambda) \Big) 
+ {\cal R}_2^+(u) = 0 \,,
\\& 
\ii \Big( {\cal F}_2^-({\rm Op}^{bw}(\Lambda) \overline u)  
+  {\rm Op}^{bw}(\Lambda) {\cal F}_2^-(\overline u) 
+ {\cal F}_2^-(\overline u) {\rm Op}^{bw}(\Lambda) \Big) + {\cal R}_2^-(\overline u) = 0\,.
\end{aligned}
\]
The claimed statement then directly follows by Lemma \ref{stime eq omologica birkhoff}. 
\end{proof}

\section{Paradifferential normal form}\label{sez forma normale paradiff}

\subsection{Paralinearization of the Schr\"odinger equation}
In this section we rewrite the equation \eqref{main} as a paradifferential system. 
From now on we shall assume the following hypothesis: 
\begin{itemize}
\item{\bf Hypothesis on local in time solutions.} There exists $\rho \gg 0$  
large enough and $T > 0$ such that
\begin{equation}\label{ansatz u}
\begin{aligned}
& u \in C^0\big([- T, T], H^\rho \big) \cap C^1([- T, T], H^{\rho - 2})\,, \\
& \sup_{t \in [- T, T]} \| u(t ) \|_\rho + \sup_{t \in [- T, T]} \| \partial_t u(t) \|_{\rho - 2} \lesssim_\rho \e\,,
\end{aligned}
\end{equation}
solves the equation \eqref{main}. 
\end{itemize}
The latter hypothesis is actually guaranteed by the local existence theorem 
proved in \cite{Feola-Iandoli-local-tori}. 
The only small difference is that the standard Laplacian 
on the torus is replaced by the more general elliptic operator \eqref{def Lambda xi}, 
but the proof can be done exactly in the same way, with the obvious small modifications. 

\begin{proposition}{\bf (Paralinearization of NLS).}\label{NLSparapara}
We have that
the equation \eqref{main}
is equivalent to the following system:
\begin{equation}\label{paralinearized equation}
\pa_{t}{U}+\ii E\opbw\big(\Lambda(\x)
\big)U+\mathcal{A}(U)U+
\mathcal{R}(U)U=0\,,\qquad U=\vect{u}{\bar{u}}\,,\;\;\;\uno:=\sm{1}{0}{0}{1}\,,
\end{equation}
where  $E$ is in \eqref{condsimpl}, $\Lambda(\x)$ is in \eqref{def Lambda xi}, 
the operator $\mathcal{A}(U)$ is in 
${\cal O}{\cal B}_\Sigma(1)$ (see Def. \ref{Op paradiff})
and has the form
\begin{equation}\label{def:opAA}
{\cal A}(U) := \ii E \opbw\begin{pmatrix}
a(U;x,\x) & b(U;x,\x) \vspace{0.2em}\\
\ov{b(U;x,-\x)} &  \ov{a(U ;x,-\x)}
\end{pmatrix}\,,
\end{equation}
where
\begin{equation*}
\begin{aligned}
a(U;x,\x)&:=\sum_{j=1}^{d}\Big[\ii (\pa_{\bar{u}u_{x_j}}f-\pa_{\bar{u}_{x_j}u}f)\x_{j}+
\frac{1}{2}\big(-\pa_{x_{j}}(\pa_{\bar{u}_{x_j}u}f)-\pa_{x_{j}}(\pa_{\bar{u}u_{x_j}}f)\big)\Big]+\pa_{u\bar{u}}f\,,
\\
b(U;x,\x)&:=-\sum_{j=1}^{d}\pa_{x_{j}}(\pa_{\bar{u}\bar{u}_{x_j}}f)+\pa_{u\bar{u}}f\,.
\end{aligned}
\end{equation*}
The remainder 
$\mathcal{R}(U)$ is a matrix of smoothing operators in the class $\mathcal{O}{S}_1(N)$.
Finally the operator $\mathcal{A}(U)$ is Hamiltonian  
according to Definition \ref{selfi}.
\end{proposition}
\begin{proof}
 By paralinearizing the nonlinearity ${\cal Q}$ in \eqref{nonlinQQQ}, using
 the Bony paralinearization formula 
(see  \cite{Metivier}, \cite{Tay-Para}) and recalling the assumption \eqref{derord2}, 
one obtains that
\begin{equation}\label{primaespQQ}
\begin{aligned}
\mathcal{Q}(u,\bar{u})&=\opbw(\pa_{u\bar{u}}f)[u]+\opbw(\pa_{\bar{u}\bar{u}}f)[\bar{u}]
+\sum_{j=1}^{d}\Big(\opbw(\pa_{\bar{u}u_{x_j}}f)[u_{x_{j}}]
+\opbw(\pa_{\bar{u}\bar{u}_{x_j}}f)[\bar{u}_{x_{j}}]\Big)
\\&
-\sum_{j=1}^{d}\pa_{x_j}\Big[
\opbw(\pa_{\bar{u}_{x_j}u}f)[u]
+\opbw(\pa_{\bar{u}\bar{u}_{x_j}}f)[\bar{u}]
\Big]+\mathcal{R}(u,\bar{u})\,,
\end{aligned}
\end{equation}
where the remainder ${\cal R}(u, \bar u)$ is smoothing, namely it satisfies 
\begin{equation}\label{albero1}
\| {\cal R}(u, \bar u)\|_{s + N} \lesssim_{s, N} \| u \|_{\rho} \| u \|_s
\end{equation}
for any $N \geq 1$, for some $\rho= \rho_N > N$ and $s \geq \rho$. 
Recall that $\pa_{x_{j}}=\opbw(\ii\x_{j})$, $j=1,\ldots,d$. Therefore,
using Proposition \ref{prop:composit} and Lemma \ref{teoremadicomposizione}, 
we have
\[
\begin{aligned}
\opbw(\pa_{\bar{u}u_{x_j}}f)&[u_{x_{j}}]-
\pa_{x_j}\opbw(\pa_{\bar{u}_{x_j}u}f)[u]=
\\&
=\opbw(\pa_{\bar{u}u_{x_j}}f)\circ\opbw(\ii \x_{j})[u]
-\opbw(\ii\x_{j})\circ\opbw(\pa_{\bar{u}_{x_j}u}f)[u]
\\&
=\opbw\Big(\ii (\pa_{\bar{u}u_{x_j}}f-\pa_{\bar{u}_{x_j}u}f)\x_{j}\Big)[u]
+\opbw\Big(
\frac{1}{2\ii}\{\pa_{\bar{u}u_{x_j}}f,\ii\x_{j}\}-
\frac{1}{2\ii}\{\ii\x_{j}, \pa_{\bar{u}_{x_j}u}f\}
\Big)[u]
\\&
=\opbw\Big(\ii (\pa_{\bar{u}u_{x_j}}f-\pa_{\bar{u}_{x_j}u}f)\x_{j}\Big)[u]
+\frac{1}{2}\opbw\Big(-\pa_{x_{j}}(\pa_{\bar{u}_{x_j}u}f)-\pa_{x_{j}}(\pa_{\bar{u}u_{x_j}}f)\Big)[u]\,,
\end{aligned}
\]
up to some smoothing remainder satisfying \eqref{albero1}.
Reasoning similarly we deduce 
\[
\begin{aligned}
\opbw(\pa_{\bar{u}\bar{u}_{x_j}}f)[\bar{u}_{x_{j}}]
&-\pa_{x_j}\opbw(\pa_{\bar{u}\bar{u}_{x_j}}f)[\bar{u}]=
\\&
=\opbw(\pa_{\bar{u}\bar{u}_{x_j}}f)\circ\opbw(\ii\x_{j})[\bar{u}]
-\opbw(\ii\x_{j})\circ\opbw(\pa_{\bar{u}\bar{u}_{x_j}}f)[\bar{u}]
\\&
=-\opbw\Big(\pa_{x_{j}}(\pa_{\bar{u}\bar{u}_{x_j}}f)\Big)[\bar{u}]\,,
\end{aligned}
\]
up to some smoothing remainder satisfying \eqref{albero1}.
Therefore, by \eqref{primaespQQ}, we obtained
\[
\begin{aligned}
\mathcal{Q}(u,\bar{u})&=\sum_{j=1}^{d}
\opbw\Big(\ii (\pa_{\bar{u}u_{x_j}}f-\pa_{\bar{u}_{x_j}u}f)\x_{j}\Big)[u]
\\&
+\opbw\Big(\pa_{u\bar{u}}f
-\frac{1}{2}\sum_{j=1}^{d}\big(\pa_{x_{j}}(\pa_{\bar{u}_{x_j}u}f)
+\pa_{x_{j}}(\pa_{\bar{u}u_{x_j}}f)\big)\Big)[u]
\\&
+\opbw\Big(\pa_{\bar{u}\bar{u}}f-\sum_{j=1}^{d}\pa_{x_{j}}(\pa_{\bar{u}\bar{u}_{x_j}}f)\Big)[\bar{u}]
+\mathcal{R}(u,\bar{u})\,,
\end{aligned}
\]
where $\mathcal{R}(u,\bar{u})$ is some remainder satisfying \eqref{albero1}.
By writing \eqref{main} as a system in $U = (u, \bar u)^{T}$, 
one gets the \eqref{paralinearized equation}-\eqref{def:opAA}.
By an explicit computation one can check that 
the operator $\mathcal{A}(U)$ is Hamiltonian. 
\end{proof}

\subsection{Diagonalization up to a smoothing remainder}
In this section  we analyze the para-differential operator 
\begin{equation*}
{\cal P}(U) := \partial_t  + \ii E {\rm Op}^{bw}(\Lambda(\x)) + {\cal A}(U)
\end{equation*}
where $\mathcal{A}(U)$ is in \eqref{def:opAA} and
 the symbol $\Lambda(\x)$ is in \eqref{def Lambda xi}.
We prove the following  result.
\begin{proposition}\label{prop off diagonal}
Let $N \in \N$, $s_0 \gg d/2$. 
Then there exists $\rho = \rho_N > N, s_0$ 
large enough such that if \eqref{ansatz u} holds, then the following holds. 
There exists a linear symplectic invertible transformation $\Phi^{(1)}(U) : {\bf H}^{s}\to{\bf H}^{s}$ 
such that 
\begin{equation}\label{killed off diag}
{\cal P}^{(1)}(U) := \Phi^{(1)}(U)^{- 1}{\cal P}(U) \Phi^{(1)}(U) 
= \partial_t + \ii E {\rm Op}^{bw}(\Lambda(\x)) + {\cal A}^{(1)}(U) 
+ {\cal R}^{(1)}(U)
\end{equation}
where 
\[
{\cal A}^{(1)}(U) = \ii \opbw\begin{pmatrix}
a^{(1)}(U; x, \xi) & 0 \\
0 & -a^{(1)}(U; x, - \xi)
\end{pmatrix}
\]
with $a^{(1)}(U;x,\x)$ a real symbol in the class $\Sigma_1^1$ and 
the remainder ${\cal R}^{(1)}(U)\in {\cal S}(\rho, N)$ 
is Hamiltonian.
Moreover, for any $s \geq \rho$, one has
\[
\| \Phi^{(1)}(U)^{\pm 1} - {\rm Id} \|_{{\cal L}(H^s)} \lesssim_s \| u \|_\rho\,. 
\]
\end{proposition}

\begin{proof}
The proposition is proved by means of an iterative procedure. 
At the step $n$ of such a procedure one has an operator 
\[
{\cal P}_n(U) = \partial_t + \ii E {\rm Op}^{bw}(\Lambda(\x)) + {\cal A}_n(U) 
+ {\cal B}_n(U) + {\cal R}_n(U)\,,
\]
where 
\begin{align}
&{\cal A}_n(U) = \ii  \opbw\begin{pmatrix}
a_n(U; x, \xi) & 0 \\
0 &-a_n(U; x, - \xi)
\end{pmatrix}\,,\label{albero22} 
\\&
 {\cal B}_n(U) := \ii \opbw\begin{pmatrix}
0 & b_n(U; x, \xi) \\
- \overline{b_n(U; x, - \xi)} & 0
\end{pmatrix}\,,\label{albero23}
\end{align}
$a_n\in \Sigma_1^1$ and $b_n \in \Sigma_1^{- n}$. 
The remainder ${\cal R}_n(U)$ is a linear Hamiltonian 
operator, smoothing of order $- N$ in the class ${\cal S}(\rho, N)$ 
for some $\rho \equiv \rho_N > N$ large enough.  
We consider 
\[
\Phi_n(U) := {\rm exp}(\ii \Psi_n(U)) 
\]
where $\Psi_n(U)$ is a para-differential operator of the form  
\[
\Psi_n(U) := \ii \opbw\begin{pmatrix}
0 & \psi_n(U; x, \xi) \\
- \overline{\psi_n(U; x, - \xi)}  & 0
\end{pmatrix}\,,
\]
where $\psi_n(U; x, \xi)$ is a symbol of order $- n - 2$ which has to be determined appropriately. 
Notice that, for any $n$, the map $\Phi_n(U)$ has the same  form of 
$\Phi_{\psi}(U)$
in \eqref{off diag astratto}. Hence it is well-posed and symplectic 
by Lemma \ref{well-posflusso2}.
We actually choose the symbol $\psi_n(U; x, \xi)$ in such a way that 
\begin{equation}\label{eq homo off diag}
\begin{aligned}
& - 2 \Lambda(\xi) \psi_n(U; x, \xi) + b_n(U; x, \xi) = 0\,, 
\quad \text{hence we set} \quad
\psi_n(U; x, \xi) := \dfrac{b_n(U; x, \xi)}{ 2 \Lambda(\xi)}\,. 
\end{aligned}
\end{equation}
Clearly, since $b_n \in \Sigma^{- n}_1$, then 
$\psi_n \in \Sigma_1^{- n - 2}$. 
Since $\Phi_n(U)$ is symplectic the transformed operator 
\[
{\cal P}_{n + 1}(U) = \Phi_n(U)^{- 1} {\cal P}_n(U) \Phi_n(U)
\]
is Hamiltonian. 
By Lemma \ref{proprieta coniugio off diagonal}, 
and using \eqref{eq homo off diag}, one gets that 
\[
\begin{aligned}
{\cal P}_{n + 1}(U) & = \partial_t + \ii E {\rm Op}^{bw}(\Lambda(\x)) 
+ {\cal A}_{n + 1}(U) + {\cal B}_{n + 1}(U) + {\cal R}_{n + 1}(U)
\end{aligned}
\]
where ${\cal A}_{n + 1}(U)$,  ${\cal B}_{n + 1}(U)$
are operators of the  form \eqref{albero22}, \eqref{albero23} with $n\rightsquigarrow n+1$
for some symbols  $a_{n + 1} \in\Sigma_1^1$  and $b_{n + 1} \in \Sigma_1^{- n - 1}$. 
Moreover the symbol $a_{n+1}$ is real valued by Lemma \ref{proprieta coniugio off diagonal}.
The remainder ${\cal R}_{n + 1}(U)$ is a linear Hamiltonian operator, 
smoothing of order $- N$ in the class ${\cal S}(\rho, N)$ 
for some $\rho \equiv \rho_N > N$ large enough. 
The proof of the lemma is then concluded. 
\end{proof}

\subsection{Normal form on the diagonal term}
In this section we prove a normal form theorem on the operator 
${\cal P}^{(1)}(U)$ in \eqref{killed off diag} 
which is a para-differential version of the normal form procedure developed in \cite{BLM-growth}. Moreover 
we denote by $(\xi; k)$ the scalar product induced 
by the matrix $G$ (see \eqref{matmetrica}), namely $(\xi; k) := G \xi \cdot k$. 
We start with the following definition. 

\begin{definition}{\bf (Normal form symbols).}\label{simboli forma normale}
A symbol $z(x, \xi)$ in ${\cal N}^m_s$ is said to be \emph{in normal form} 
(with parameters $\delta, \ep, \tau$) if 
\[
\begin{gathered}
z( x, \xi) = \sum_{k \in \Z^d} \widehat z( k, \xi) e^{\ii k \cdot x}\
\end{gathered}
\]
satisfies
$$
\widehat z( k, \xi)\neq 0 \quad \Longrightarrow \quad |(\xi; k)| \leq 
\langle \xi  \rangle^{\delta}| k|^{-\tau} \textrm{ and } |k |\leq \langle \xi \rangle^{\ep}\,
$$
for any $k \neq 0$, $\xi \in \R^d$.
\end{definition}
We shall fix appropriately the parameters $\e, \delta \in (0, 1)$, $\tau > 0$ as  
\begin{equation}\label{legami.tra.parametri}
\frac23 < \delta < 1\,, \quad \tau > d - 1,   \quad 0 < \epsilon < \frac{\delta}{\tau + 1}
\end{equation}
cf. \cite{comm-pdes}. The main result of this section is the following.

\begin{proposition}\label{BLM-paradiff}
Let $N \in \N$, $s_0 \gg d/2$. 
Then there exists $\rho = \rho_N > N, s_0$ 
large enough such that if \eqref{ansatz u} holds, then the following holds. There exists a linear symplectic invertible transformation 
$\Phi^{(2)}(U) : {\bf H}^{s}\to {\bf H}^{s}$ such that
\begin{equation*}
{\cal P}^{(2)}(U) := \Phi^{(2)}(U)^{- 1}{\cal P}^{(1)}(U) \Phi^{(2)}(U) 
= \partial_t + \ii E {\rm Op}^{bw}(\Lambda(\x)) + {\cal Z}(U) + {\cal R}^{(2)}(U)
\end{equation*}
where $\mathcal{P}_1(U)$ is in \eqref{killed off diag}, the operator $\mathcal{Z}(U)$
has the form
\begin{equation}\label{fn Z BLM}
{\cal Z}(U) = \ii \opbw\begin{pmatrix}
z(U; x, \xi) & 0 \\
0 & -z(U; x, - \xi)\,,
\end{pmatrix}
\end{equation}
with $z(U;x,\x)$ a real symbol in the class $\Sigma_1^1$ 
which is in normal form, according to the Definition \ref{simboli forma normale}. 
The remainder ${\cal R}^{(2)}(U)$ is Hamiltonian 
and it is in the class ${\cal S}(N)$. Moreover, for any $s \geq \rho$,
one has
\[
\| \Phi^{(2)}(U)^{\pm 1}\|_{{\cal L}(H^s)} \leq 1 + C(s) \| u \|_\rho\,.
\]
\end{proposition}

In order to prove the proposition stated above, 
we need some further symbolic calculus.
First of all, we consider an even smooth cut-off function 
$\chi: \R \rightarrow [0, 1]$ 
with the property that 
{$\chi(y) = 1$} for all $y$ with 
$|y| \leq\tfrac{1}{2}$ and {$\chi(y) = 0$} 
for all $y$ with $|y| \geq 1$.

\begin{definition}\label{rmk dec symb}
Given $\ep, \delta , \tau$ as in \eqref{legami.tra.parametri}, define the following functions:
\begin{align*}
\chi_k(\xi) &= 
\chi\left(\frac{2 |k|^\tau (\xi ; k) }{\langle\xi \rangle^\delta}\right)\,, 
\;\quad  k \in \Z^d\setminus\{0\},
\\
\tilde{\chi}_k(\xi) &= \chi \left(\frac{|k|}{\langle \xi 
\rangle^\e}\right)\, , \qquad \qquad k \in \Z^d\setminus \{0\}\ .
\end{align*}
Correspondingly, given a symbol $a \in {\cal N}^m_s$, 
we decompose it as follows:
\begin{equation*}
a = \langle a \rangle +  a^{(\rm nr)} + a^{(\rm res)} + a^{(S)}\,,
\end{equation*}
where $\langle a \rangle$ is the $x$-average of the symbol of $a\,,$
namely
\begin{equation*}
\langle a \rangle ( \xi) = \frac{1}{\mu( \T^d)}\int_{\T^d} a ( x, \xi)\ d x\,,
\end{equation*}
and
\begin{equation}\label{decomp} 
\begin{aligned}
a^{(\rm res)}(x, \xi) &= 
\sum_{k \neq 0}  \chi_k(\xi) \tilde{\chi}_k(\xi) \widehat a( k, \xi) e^{\ii k \cdot x}\,,
\\
a^{(\rm nr)}(  x, \xi) &= 
\sum_{k \neq 0}  \left(1 - \chi_k(\xi)\right) \tilde{\chi}_k(\xi) \widehat a( k, \xi) e^{\ii k \cdot x}\,,
\\
a^{(S)}(  x, \xi) &= 
\sum_{k \neq 0}  \left(1 - \tilde{\chi}_k(\xi)\right) \widehat a( k, \xi) e^{\ii k \cdot x}\,.
\end{aligned}
\end{equation}
We also define 
\begin{equation}\label{sol equazione omologica NOI}
g_a(x, \xi) := - \sum_{k \neq 0} \frac{1}{2 (\xi; k)}
\left(1 - \chi_k(\xi)\right) \tilde{\chi}_k(\xi) 
\widehat a( k, \xi) e^{\ii k \cdot x}\,. 
\end{equation}
\end{definition}

In \cite{comm-pdes}, Lemma 5.4, we provided suitable bounds 
for the cut-off functions defined above. 

\begin{lemma}
For any muti-index $\alpha \in \N^d$, one has 
\begin{equation}\label{pastiera}
\begin{aligned}
& |\chi_k(\xi)| \leq 1 \,, 
\quad |\partial_\xi^\alpha \chi_k(\xi)| \lesssim_\alpha 
|k|^{(\tau + 1)|\alpha|}\langle \xi \rangle^{-\delta |\alpha|}\,,
\\& 
|\tilde \chi_k(\xi)| \leq 1 \,, 
\quad |\partial_\xi^\alpha \tilde \chi_k(\xi)| \lesssim_\alpha  
|k|^{|\alpha|}\langle \xi \rangle^{-(\e + |\alpha|)}\,,
\end{aligned}
\end{equation}
and 
\begin{equation}\label{cut off dk}
d_k(\xi) := \frac{1}{2 (\xi ; k)} (1 - \chi_k(\xi))\,, 
\quad 
|\partial_\xi^\alpha d_k(\xi)| \lesssim_\alpha 
\frac{\langle k \rangle^{(|\alpha| + 1) \tau 
+ |\alpha|}}{ \langle \xi \rangle^{\delta(|\alpha| + 1)}} \,. 
\end{equation}
As a consequence, for any $s \geq 0$ and $k \in \Z^d$, one has 
\begin{equation}\label{stima chi k tilde k}
|\chi_k|_{0, s} \lesssim_s \langle k \rangle^{(\tau + 1)s}\,, 
\quad 
|\tilde \chi_k|_{0, s} \lesssim_s \langle k \rangle^s\,, 
\quad 
|d_k|_{- \delta, s} \lesssim_s  \langle k \rangle^{(s + 1) \tau + s}\,.
\end{equation}
\end{lemma}
We now prove the following Lemma.

\begin{lemma}\label{prop splitting cal N ms}
Let $m \in \R$. Then, for any $s \geq 0$, the linear map ${\cal N}^m_s \to {\cal N}^m_s$, 
$a \mapsto \langle a \rangle$ is linear and continuous.  
For any $s \geq 0$, there exists $\sigma_s > s$ large enough such that the maps
(see Def. \ref{rmk dec symb})
\[
\begin{aligned}
& {\cal N}^m_{\sigma_s} \to {\cal N}^m_s\,, 
\qquad a \mapsto a^{\rm nr}\,, 
\quad a \mapsto a^{\rm res}\,, 
\\& 
{\cal N}^m_{\sigma_s} \to {\cal N}^{m - \delta}_s\,, 
\quad a \mapsto g_a\,,
\end{aligned}
\]
are linear and continuous. 
Let $N \in \N$, Then there exists $\rho = \rho_N > 0$ large enough, 
such that, for any $s \geq \rho$, the map 
\[
{\cal N}^{m}_{\rho} \to {\cal B}(H^s, H^{s + N})\,, 
\quad 
a \mapsto {\cal R}^{(S)}(a) := {\rm Op}^{bw}(a^{(S)})\,,
\]
is linear and continuous. 
\end{lemma}

\begin{proof}
Since $\langle a \rangle$ is only the space average of the symbol $a$ 
it is straightforward that 
$|\langle a \rangle|_{m, s} \lesssim |a|_{m, s}$. 
We now estimate $g_a$ in terms of $a$. 
The estimates for $a^{\rm nr}$ and $a^{\rm res}$ can be done arguing similarly. 
By the definitions \eqref{sol equazione omologica NOI}, \eqref{cut off dk}, 
one has that $g_a$ can be written as 
\[
g_a(x, \xi) = - \sum_{k \neq 0} d_k(\xi) \tilde{\chi}_k(\xi) \widehat a( k, \xi) e^{\ii k \cdot x}\,.
\]
By applying Lemma \ref{non hom symbols Fourier} 
one gets, for any $s \geq 0$, $N \in \N$, for any $k \in \Z^d$, that
\begin{equation}\label{stima aq nel lemma}
|\widehat a(k,  \cdot)|_{m, s} \lesssim_{ N} 
\langle k \rangle^{- N}|a|_{m, s + N}\,. 
\end{equation}
Fix 
\[
N :=  (s + 1)\tau + 2 s + d + 1\,.
\]
One has that 
\[
\begin{aligned}
|d_k \tilde \chi_k \widehat a(k, \cdot)|_{m - \delta, s} 
&  
\lesssim_s |d_k|_{- \delta, s}|\tilde{\chi}_k|_{0, s} |\widehat a( k, \cdot)|_{m, s} 
\stackrel{\eqref{stima chi k tilde k}, \eqref{stima aq nel lemma}}{\lesssim_s} 
\langle k \rangle^{(s + 1) \tau + 2 s  - N} |a|_{m, s + N} 
\\& 
\lesssim_s \langle k \rangle^{- d - 1} |a|_{m, s + N}\,.
\end{aligned}
\]
Using that $\sum_{k} \langle k \rangle^{- d - 1}$ is convergent, 
one then gets that 
$|g_a|_{m - \delta, s} \lesssim_s |a|_{m, s + N}$
and the claimed statement follows. 
		
\noindent
We now estimate the symbol $a^{(S)}$. 
By the definition of the cut off function $\tilde \chi_k$ in Def. \ref{rmk dec symb}, 
one has that 
\[
{\rm supp}(1- \tilde \chi_k) \subseteq 
\big\{ \xi : \langle \xi \rangle^\e \leq 2 |k| \big\}\,,
\]
hence, by the estimate \eqref{pastiera}, 
one has that, for any $N \in \N$, $\alpha \in \N^d$,
\[
\langle \xi \rangle^{N + m} |1 - \tilde \chi_k(\xi)| 
\lesssim \langle k \rangle^{(N + m)/\e}\,, 
\quad 
\langle \xi \rangle^{N + m + |\alpha|} |\partial_\xi^\alpha(1 - \tilde \chi_k(\xi))| 
\lesssim_\alpha \langle k \rangle^{|\alpha| + \frac{N + m + |\alpha|}{\e}}\,,
\]
implying that 
\begin{equation}\label{pastiera 1}
|1 - \tilde \chi|_{- N - m, s} 
\lesssim_s \langle k \rangle^{s + \frac{N + m + s}{\e}}\,. 
\end{equation}
Now fix 
\[
M := d + 1 + s + (N+m+s)\e^{-1}\,.
\]
By Lemma  \ref{non hom symbols Fourier}, one has 
\[
\begin{aligned}
|(1 - \tilde \chi_k) \widehat a(k, \cdot)|_{- N, s} & 
\lesssim |1 - \tilde \chi_k|_{- N - m, s} |\widehat a(k, \cdot)|_{m, s} 
\stackrel{\eqref{pastiera 1}}{\lesssim_s} 
\langle k \rangle^{s + \frac{N + m + s}{\e} - M} |a|_{m, s + M} 
\\& 
\lesssim_s \langle k \rangle^{- d - 1} |a|_{m, s + M}\,.
\end{aligned}
\]
Hence, using that $\sum_{k} \langle k \rangle^{- d - 1}$ is convergent, 
one gets that 
\begin{equation}\label{stima a S a nella dim}
|a^{(S)}|_{- N, s} \lesssim_s |a|_{m, s + M}\,.
\end{equation} 
We now consider the operator ${\cal R}^{(S)}(a) := {\rm Op}^{bw}(a^{(S)})$. 
Its action is given by 
\[
\begin{aligned}
{\cal R}^{(S)}(a)[w] & = 
\sum_{k, \xi \in \Z^d} \eta_\e \Big(\frac{|k - \xi|}{\langle k + \xi \rangle} \Big) 
\widehat a^{(S)}\Big(k - \xi, \frac{k + \xi}{2} \Big) \widehat u(\xi) e^{\ii k \cdot x} 
\\& 
\stackrel{\eqref{decomp}}{=} 
\sum_{k, \xi \in \Z^d} \eta_\e \Big(\frac{|k - \xi|}{\langle k + \xi \rangle} \Big) 
\Big(1 - \tilde{\chi}_k\Big( \frac{k + \xi}{2} \Big) \Big) 
\widehat a\Big(k - \xi, \frac{k + \xi}{2} \Big) \widehat u(\xi) e^{\ii k \cdot x} \,.
\end{aligned}
\]
Clearly the map $a \mapsto {\cal R}^{(S)}(a)$ is linear. 
By defining 
\[
\rho := s_0 + M = s_0 + d + 1 + s_0 + \frac{N + m + s_0}{\e}\,, 
\quad s_0 := \frac{d}{2} + 1
\]
the estimate \eqref{stima a S a nella dim} reads 
$|a^{(S)}|_{- N, s_0} \lesssim_N |a|_{m, \rho}$ 
and therefore, by applying Lemma \ref{standard cont bony}, 
one has that, 
for any $s \geq \rho$, 
\[
\| {\cal R}^{(S)}(a) \|_{{\cal L}(H^s, H^{s + N})} 
\lesssim_s |a^{(S)}|_{- N, s_0} 
\lesssim_{s, N} |a|_{m, \rho}\,.
\]
Hence the linear map 
\[
{\cal N}^m_\rho \to {\cal L}(H^s, H^{s + N})\,, 
\quad a \mapsto {\cal R}^{(S)}(a)
\]
is bounded. The claimed statement has then been proved. 
\end{proof}

\begin{lemma}\label{corollario splitting simboli non lineari}
Let $a \in \Sigma^m_1$. 
Then 
$\langle a \rangle, a^{(\rm nr)}, a^{(\rm res)} \in \Sigma^m_1$ 
and $g_a \in \Sigma^{m - \delta}_1$. 
Moreover, for any $N \in \N$, the remainder
${\cal R}^{(S)}(U) := {\rm Op}^{bw}\Big( a^{(S)}(U; x, \xi) \Big)$ 
belongs to the class ${\cal S}( N)$.
\end{lemma}

\begin{proof}
Let $a \in \Sigma^m_1$. We show that 
$g_a \equiv g_a \in \Sigma_1^{m - \delta}$. 
The proof that $a, a^{(\rm nr)}, a^{(\rm res)} \in \Sigma^m_1$ is analogous. 
Since $a \in \Sigma^m_1$, then 
\begin{equation}\label{splitting pastiera}
a = a_l + a_q, \quad \text{with} 
\quad a_l \in O^m_1 
\quad 
\text{and} 
\quad a_q \in \Gamma^m_2
\end{equation} 
and according to the definitions  
\eqref{sol equazione omologica NOI}, \eqref{cut off dk}, 
one obtains a corresponding splitting $g = g_l + g_q$ where 
\[
\begin{aligned}
& g_{a_l} (x, \xi) = 
- \sum_{k \neq 0} d_k(\xi) \tilde \chi_k(\xi) \widehat a_l(U; x, \xi) e^{\ii k \cdot x}\,, 
\\& 
g_{a_q}(x, \xi) =
- \sum_{k \neq 0} d_k(\xi) \tilde \chi_k(\xi) \widehat a_q(U; x, \xi) e^{\ii k \cdot x} \,. 
\end{aligned}
\]
We show that $g_{a_l} \in O^m_1$ and $g_{a_q} \in \Gamma^m_2$. 
		
\noindent
Since $a_l \in O^m_1$, then 
\[
\widehat a(U; k, \xi) = m_+(k, \xi) \widehat u(k) + m_{- }(k, \xi) \overline{\widehat u(- k)}
\]
for some suitable multipliers $m_+, m_-$ and 
therefore $g_l$ is a symbol which is linear in $U$ and of the same form as $a_l$. 
It remains only to show that $g_l$ is in $\Gamma^m_1$ 
and $g_q \in \Gamma^m_2$. 
Fix $s \geq 0$ and let $\sigma_s > s$ the constant provided by Lemma \ref{prop splitting cal N ms}. 
Since $a_l \in \Gamma^m_1$, $a_q \in \Gamma^m_2$ 
there exists a constant $\sigma_s' > \sigma_s$ 
and a radius $r = r(s) \in (0, 1)$ such that the linear map 
\[
H^{\sigma_s'} \to {\cal N}^m_{\sigma_s}\,, 
\quad U \mapsto a_l(U; x, \xi)
\]
is continuous and the map 
\[
B_{\sigma_s' }(r) \to {\cal N}^m_{\sigma_s}\,, 
\quad U \mapsto a_q(U; x, \xi)
\]
is ${\cal C}^\infty$ and vanishes of order two at $U = 0$. 
Since by Lemma \ref{prop splitting cal N ms}, 
the map ${\cal N}^m_{\sigma_s} \to {\cal N}^m_s$, $a \mapsto g_a$ 
is linear and continuous, then by composition one gets that the map 
\[
B_{\sigma_s' }(r) \to {\cal N}^m_{s}\,, \quad U \mapsto g_{a_l}(U; x, \xi)
\]
is linear and continuous and the map 
\[
B_{\sigma_s' }(r) \to {\cal N}^m_{s}\,, \quad U \mapsto g_{a_q}(U; x, \xi)
\]
is ${\cal C}^\infty$ and vanishes of order two at $U = 0$. 
This shows the claimed statement. 
		
\medskip
		
\noindent
{\sc Analysis of the operator ${\cal R}^{(S)}(U) = {\rm Op}^{bw}(a^{(S)}(U; x, \xi))$.} 
According to \eqref{splitting pastiera}
\begin{equation}\label{pastiera smoothing simbolo}
\begin{aligned}
 a^{(S)} & = a^{(S)}_l + a^{(S)}_q\,, 
 \\
 a^{(S)}_l & = \sum_{k \neq 0}  \left(1 - \tilde{\chi}_k(\xi)\right) \widehat a_l( k, \xi) e^{\ii k \cdot x}  
 \\& 
 = \sum_{k \neq 0}  \left(1 - \tilde{\chi}_k(\xi)\right) m_+(k, \xi)\widehat u(k) e^{\ii k \cdot x} 
 + \sum_{k \neq 0}  \left(1 - \tilde{\chi}_k(\xi)\right) m_-(k, \xi)\overline{\widehat u(- k)} e^{\ii k \cdot x} \,, 
 \\
 a^{(S)}_q & = \sum_{k \neq 0}  \left(1 - \tilde{\chi}_k(\xi)\right) \widehat a_q( k, \xi) e^{\ii k \cdot x}\,,
\end{aligned}
\end{equation}
and correspondingly
\begin{equation*}
\begin{aligned}
& {\cal R}^{(S)}(U) = {\cal R}^{(S)}_l(U) + {\cal R}^{(S)}_q(U)\,, 
\qquad
{\cal R}^{(S)}_l(U) := {\rm Op}^{bw}(a^{(S)}_l)\,, 
\quad {\cal R}_q^{(S)}(U) := {\rm Op}^{bw}(a^{(S)}_q)\,. 
\end{aligned}
\end{equation*}
Fix $N \in \N$ and let $\rho \equiv \rho_N$ be the constant 
appearing in Lemma \ref{prop splitting cal N ms}. 
Since $a_l \in O^m_1$ and $a_q \in \Gamma^m_2$, 
one has that for some constant $\sigma_\rho > \rho$ large enough
\[
|a_l|_{m, \rho} \lesssim_\rho \| U \|_{\sigma_\rho}\,, 
\quad |a_q|_{m, \rho} \lesssim_\rho \| U \|_{\sigma_\rho}^2 \,,
\]
implying that, for any $s \geq \sigma_\rho > \rho$, one has
$$
\begin{aligned}
& \| {\cal R}_l^{(S)}(U) \|_{{\cal L}(H^s, H^{s + N})} \lesssim_s |a_l|_{m, \rho} \lesssim_{s} \| U \|_{\sigma_\rho}, \\
& \| {\cal R}_q^{(S)}(U) \|_{{\cal L}(H^s, H^{s + N})} \lesssim_s |a_q|_{m, \rho} \lesssim_{s} \| U \|_{\sigma_\rho}^2\,,
\end{aligned}
$$
and hence ${\cal R}_q^{(S)} \in {\cal S}_2(N)$. 
In order to show that ${\cal R}_l^{(S)}(U)$ belongs to the class $\mathcal{O}{\cal S}_1(N)$ 
it remains only to show that it is sum
of terms of the form \eqref{smoothing bilineare}. 
This follows since 
\[
\begin{aligned}
{\cal R}^{(S)}_l(U)[w] & = 
\sum_{k, \xi \in \Z^d}\eta_\e \Big(\frac{|k - \xi|}{\langle k + \xi \rangle} \Big) 
\Big(1 - \tilde{\chi}_k\Big( \frac{k + \xi}{2} \Big) \Big) 
\widehat a_l\Big(k - \xi, \frac{k + \xi}{2} \Big) \widehat w(\xi) e^{\ii k \cdot x} 
\\& 
\stackrel{\eqref{pastiera smoothing simbolo}}{=} 
\sum_{k, \xi \in \Z^d} r_+(k, \xi)\widehat u(k - \xi) \widehat w(\xi) e^{\ii k \cdot x} 
+ \sum_{k, \xi \in \Z^d} r_-(k, \xi)\overline{\widehat u( \xi - k))} \widehat w(\xi) e^{\ii k \cdot x}
\end{aligned}
\]
where 
\[
r_{\pm}(k, \xi)  := \eta_\e \Big(\frac{|k - \xi|}{\langle k + \xi \rangle} \Big) 
\Big(1 - \tilde{\chi}_k\Big( \frac{k + \xi}{2} \Big) \Big) m_{\pm}\Big(k - \xi, \frac{\xi + k}{2}\Big)\,.
\]
The claimed statement has then been proved. 
\end{proof}
We are now in position to prove the Proposition \ref{BLM-paradiff}. 
		
\begin{proof}[{\bf Proof of Proposition \ref{BLM-paradiff}}]
The Proposition is proved also inductively, 
hence we describe the induction step of the procedure. In the proof it is convenient to use the following notations. If ${\mathcal O}$ is one of the classes of operators defined in Section \ref{sez def simpboli e smoothing}, we write $A = B + {\cal O}$ if $A - B$ belongs to the class ${\cal O}$. 

\noindent
We define the {\it gain of regularization} along the reduction procedure as
\begin{equation}\label{def frak e}
\mathfrak e := {\rm min}\big\{\delta,  3 \delta - 2, 2 \delta - 1 \big\}\,.
\end{equation} 
At the $n$-th step, we deal with a Hamiltonian para-differential operator of the form 
\begin{equation}\label{calPPnn}
{\cal P}_n^{(1)}(U)  := 
\partial_t + \ii E {\rm Op}^{bw}(\Lambda(\x)) +  {\cal Z}_n(U) + {\cal A}_n(U) + {\cal R}_n(U)
\end{equation}
where
\begin{equation}\label{cal Pn (1)}
\begin{aligned}
{\cal Z}_n(U) & := 
\ii \opbw\begin{pmatrix}
z_n(U; x, \xi)  & 0 \\
0 & -  z_n(U; x, -\xi) 
\end{pmatrix}\,, \quad z_n \in \Sigma_1^1\,,\quad 
z_n \quad \text{is real and in normal form}\,, 
\\
{\cal A}_n(U) & :=\ii\opbw\begin{pmatrix}
a_n(U; x, \xi)  & 0 \\
0 & - a_n(U; x, -\xi) 
\end{pmatrix}\,, 
\quad a_n \in \Sigma_1^{1 - n {\mathfrak e}}, 
\quad a_n \quad \text{is real}\,, 
\\
{\cal R}_n& \in {\cal S}( N)\,. 
\end{aligned}
\end{equation}
By Lemma \ref{corollario splitting simboli non lineari}, 
one has that 
\begin{equation*}
\begin{aligned}
&{\rm Op}^{bw}(a_n)  = 
{\rm Op}^{bw}(\langle a_n \rangle + a_n^{(\rm nr)} + a_n^{(\rm res})) 
+ {\cal R}^{(S)}(a_n)\,, 
\\& 
\langle a_n \rangle, a_n^{(\rm nr)}, a_n^{(\rm res)} \in \Sigma_1^{1 - n \mathfrak e}\,, 
\quad {\cal R}^{(S)}(a_n) \in {\cal S}(N)\,.
\end{aligned}
\end{equation*}
Moreover, by defining (as in \eqref{sol equazione omologica NOI})
\begin{equation*}
g_n(U; x, \xi) := - \sum_{k \neq 0} \frac{1}{2 (\xi; k)}
\left(1 - \chi_k(\xi)\right) \tilde{\chi}_k(\xi) \widehat a_n(U;  k, \xi) e^{\ii k \cdot x}\,,
\end{equation*}
one has that  the symbol $g_{n}(U;x,\x)$ is in $ \Sigma_1^{1 - n \mathfrak e - \delta}$
and solves the equation
\begin{equation}\label{omologica descent noi}
\{ \Lambda, g_n \} + a_n^{(\rm nr)} = 0\,. 
\end{equation}
We then consider the map
\begin{equation*}
{\bf \Phi}_n(U) := \begin{pmatrix}
\Phi_n(U) & 0 \\
0 & \overline{\Phi_n(U)}
\end{pmatrix}
\end{equation*}
where $\Phi_n(U)$ is the time one flow map of 
\[
\partial_\tau \Phi_n(U) = \ii {\rm Op}^{bw}(g_n) \Phi_n^\tau(U)\,, 
\quad \Phi_n^0(U) = {\rm Id}\,.
\]
The map ${\bf \Phi}_n(U)$ is well-posed and symplectic by Lemma \ref{lemm:flussoGGG}.
We now compute the conjugated operator 
\[
{\cal P}_{n + 1}^{(1)}(U) := {\bf \Phi}_n(U)^{- 1} {\cal P}_n^{(1)}(U) {\bf \Phi}_n(U)\,.
\] 
Note that for any $n \geq 0$, $1 - n \mathfrak e - \delta < \delta < 1$, 
hence the conjugation Lemmas of Section \ref{sez calcolo simbolico} can be applied. 
In particular, by applying Lemmata \ref{coniugio flusso diag}, 
\ref{coniugio partial t flusso diag} 
(where $n$ is replaced by $1 - n \mathfrak e - \delta$), 
one gets 
\begin{equation*}
\begin{aligned}
{ \Phi}_n(U)^{- 1} \partial_t {\Phi}_n(U) & =
 \partial_t + {\cal O B}_\Sigma(1 - n \mathfrak e - \delta) + {\cal S}(N)\,, 
 \\
{\Phi}_n(U)^{- 1} \ii  {\rm Op}^{bw}(\Lambda) { \Phi}_n(U) & =  
\ii {\rm Op}^{bw}(\Lambda +  \{ \Lambda, g_n \} ) 
+ {\cal O}{\cal B}_\Sigma (3 - 2 n {\mathfrak e} - 3 \delta) + {\cal S}( N)\,, 
\\
{\Phi}_n(U)^{- 1} \ii {\rm Op}^{bw}(z_n) \Phi_n(U)  & = 
\ii {\rm Op}^{bw}(z_n) + {\cal O}{\cal B}_\Sigma(2 - n \mathfrak e - 2 \delta) + {\cal S}( N)  \,, 
\\
{\Phi}_n(U)^{- 1} \ii {\rm Op}^{bw}(a_n) \Phi_n(U)  & =
\ii {\rm Op}^{bw}(\langle a_n \rangle + a_n^{(\rm nr)} + a_n^{(\rm res})) 
+  {\cal O}{\cal B}_\Sigma(2 - 2 n \mathfrak e - 2 \delta) + {\cal S}(N)\,, 
\\
{\bf \Phi}_n(U)^{- 1} {\cal R}_n(U) {\bf \Phi}_n(U) & = {\cal S}(N)\,. 
\end{aligned}
\end{equation*}
By the definition of ${\mathfrak e}$ given in \eqref{def frak e}, 
one obtains that 
\[
1 - n \mathfrak e - \delta\,, 3 - 2 n {\mathfrak e} - 3 \delta\,, 
2 - n \mathfrak e - 2 \delta\,, 2 - 2 n \mathfrak e - 2 \delta 
\leq 1 - (n + 1) \mathfrak e
\]
and using that $g_n$ solves the equation 
\eqref{omologica descent noi}, 
one obtains that ${\cal P}_{n + 1}^{(1)}(U)$ has the form 
\eqref{calPPnn} with $n\rightsquigarrow n+1$,
for some ${\cal R}_{n + 1}(U) \in {\cal S}(N)$ and 
\begin{equation*}
\begin{aligned}
{\cal Z}_{n + 1}(U) & := 
\ii \opbw\begin{pmatrix}
 z_{n + 1}(U; x, \xi)  & 0 \\
0 & - z_{n + 1}(U; x, -\xi) 
\end{pmatrix}\,, 
\quad z_{n + 1} \in \Sigma_1^1\,, 
\quad
z_{n + 1} := z_n + \langle a_n \rangle + a_{n}^{(\rm res)}\,,
\\
{\cal A}_{n + 1}(U) & := 
\ii \opbw\begin{pmatrix}
a_{n + 1}(U; x, \xi) & 0 \\
0 & - a_{n + 1}(U; x, -\xi) 
\end{pmatrix}\,, 
\quad a_{n + 1} \in \Sigma_1^{1 - (n + 1) {\mathfrak e}}\,.
\end{aligned}
\end{equation*}
%
Since $\Phi_n(U)$  is a linear symplectic map, 
the paradifferential operator ${\cal P}_{n + 1}(U)$is Hamiltonian, 
hence $z_{n + 1}$ and $a_{n + 1}$ are real symbols. 
Furthermore, $z_{n + 1}$ is a symbol in normal form, 
since $z_n$ is in normal form by the induction hypothesis and 
$\langle a_{n}\rangle, a_{n}^{({\rm res})}$ are in normal form by their definition. 
The claimed induction statement has then been proved. 
\end{proof}
		
\section{The Birkhoff normal form step}\label{sec:BNFstep}
By Propositions \ref{prop off diagonal}, \ref{BLM-paradiff}, 
one has that $U$ solves the equation \eqref{paralinearized equation} 
if and only if $U := \Phi^{(1)}(U) \Phi^{(2)}(U) W$ solves 
\begin{equation*}
{\cal P}^{(3)}(U) [W] = 0\,, 
\quad {\cal P}^{(3)}(U) := 
\partial_t  + \ii E {\rm Op}^{bw}(\Lambda(\x)) W + {\cal Z}(U)  + {\cal Q}(U)  
\end{equation*}
where ${\cal Q} \in {\cal S}(N)$ and ${\cal Z}(U)$ 
is the normal form operator provided in Proposition \ref{BLM-paradiff}. 
We now perform a step of Birkhoff normal in order 
to remove the quadratic terms from ${\cal Q}(U) W$. 
Since ${\cal Q} \in {\cal S}(N)$, then 
\begin{equation}\label{decomp quadratic smoothing}
{\cal Q} = {\cal Q}_l + {\cal Q}_q\,, \quad {\cal Q}_l \in {\cal O}{\cal S}_1(N)\,, 
\quad {\cal Q}_q \in {\cal S}_2(N)\,. 
\end{equation} 
We fix the number of regularization step $N$ as 
\begin{equation}\label{N fissato}
N := \tau + 3\,,
\end{equation}
where $\tau$ is the loss of derivatives in the small 
divisors estimate of Lemma \ref{condizioni di non-risonanza}. 
We prove the following.

\begin{proposition}\label{BNF step}
Let $\mathcal{G}\in (0,+\infty)$ be the full Lebesgue measure set
given by Lemma \ref{condizioni di non-risonanza}. 
Then for any $m\in \mathcal{G}$ 
the following holds.
Then there exists $\rho \equiv \rho(\tau) \gg 0$ large enough such that if \eqref{ansatz u} is fullfilled, then the following holds. 
There exists  a linear and invertible transformation 
$\Phi^{(3)}(U) : {\bf H}^{s}\to {\bf H}^{s}$ 
such that 
\begin{equation}\label{op dopo birkhoff}
{\cal P}^{(4)}(U) := \Phi^{(3)}(U)^{- 1}{\cal P}^{(3)}(U) \Phi^{(3)}(U) = 
\partial_t + \ii E {\rm Op}^{bw}(\Lambda(\x)) + {\cal Z}(U) + {\cal R}^{(4)}(U)
\end{equation}
where ${\cal Z}(U)$ is given in Proposition \ref{BLM-paradiff} and 
${\cal R}^{(4)}(U) W$ is cubic and one-smoothing remainder, 
namely it satisfies, for any $s \geq \rho$, $W \in H^s$, 
the estimate
\begin{equation}\label{stima cal R (4)}
\| {\cal R}^{(4)}(U)W \|_{s + 1} \lesssim_s \| U \|_\rho^2 \| W \|_s\,.
\end{equation}
Moreover, for any $s \geq \rho$, one has
\[
\| \Phi^{(3)}(U)^{\pm 1}\|_{{\cal L}(H^s)} \leq 1 + C(s) \| u \|_\rho\,. 
\]
\end{proposition}

\begin{proof}
We look for a smoothing operator ${\cal F} \in {\cal O}{\cal S}_1(3)$ 
and we consider the flow map 
$\Phi_{\cal F}^\tau(U)$. 
We then set $\Phi^{(3)}(U) := \Phi_{\cal F}^1(U)$. 
By applying Lemma \ref{coniugazioni flusso smoothing}, one gets that 
\begin{equation}\label{coniugi lemma smoothing 1}
\begin{aligned}
\Phi^{(3)}(U)^{- 1} \circ \partial_t \circ \Phi^{(3)}(U) & = 
\partial_t - {\cal F}(\ii E {\rm Op}^{bw}(\Lambda) U) + {\cal S}_2(3)  
\\
\Phi^{(3)}(U)^{- 1} \ii E {\rm Op}^{bw}(\Lambda) \Phi^{(3)}(U) & = 
\ii E {\rm Op}^{bw}(\Lambda) + [\ii E {\rm Op}^{bw}(\Lambda), {\cal F}(U)] + {\cal S}_2(1) 
\\
\Phi^{(3)}(U)^{- 1} {\cal Z}(U) \Phi^{(3)}(U) & = {\cal Z}(U) + {\cal S}_2( 2) 
\\
\Phi^{(3)}(U)^{- 1} {\cal Q}(U) \Phi^{(3)}(U) & 
\stackrel{\eqref{decomp quadratic smoothing}, \eqref{N fissato}}{=} 
{\cal Q}_l(U) + {\cal S}_2(\tau + 3)\,.
\end{aligned}
\end{equation}
By applying Lemma \ref{eq omologica Birkhoff step}, 
since ${\cal Q}_l \in {\cal O}{\cal S}_1(\tau + 3)$, then there exists 
${\cal F} \in {\cal O}{\cal S}_1(3)$ 
which solves 
\begin{equation}\label{applicazione omologica smoothing}
- {\cal F}(\ii E {\rm Op}^{bw}(\Lambda) U) + [\ii E {\rm Op}^{bw}(\Lambda), {\cal F}(U)]  
+ {\cal Q}_l(U) = 0\,. 
\end{equation}
Then \eqref{coniugi lemma smoothing 1}, \eqref{applicazione omologica smoothing} 
imply that ${\cal P}^{(4)}(U) := \Phi^{(3)}(U)^{- 1} {\cal P}^{(3)}(U) \Phi^{(3)}(U)$ 
has the form \eqref{op dopo birkhoff}, with ${\cal R}^{(4)} \in {\cal S}_2(1)$ 
and hence satisfying the claimed estimate \eqref{stima cal R (4)}. 
\end{proof}

\section{Energy estimates and proof of Theorem \ref{teo principale} concluded}\label{sezione finale stime energia}
In this section we conclude the proof of the main result of the paper, 
namely Theorem \ref{teo principale}. 
The main point is to provide an energy estimate for the reduced equation. 
\begin{equation}\label{bla bla eq per stima di energia}
\partial_t W + \ii E {\rm Op}^{bw}(\Lambda(\x)) W + {\cal Z}(U) W + {\cal R}^{(4)}(U) W = 0\,.
\end{equation}
First of all, let us consider the linear flow associated to the normal form equation
\begin{equation*}
\partial_t W + \ii E {\rm Op}^{bw}(\Lambda(\x)) W + {\cal Z}(U) W = 0\,,
\end{equation*}
which, by \eqref{fn Z BLM}, is equivalent to the scalar equation 
\begin{equation}\label{NF equation scalare}
\partial_t w + \ii {\rm Op}^{bw}\Big(\Lambda(\xi) + z(U; x, \xi) \Big) w = 0 \,,
\end{equation}
where $z$ is a real symbol in normal form (see Def. \ref{simboli forma normale}). 
The following Lemma is proved in \cite{BLM-growth}, Section 5.1. 

\begin{lemma}\label{NF flow Lemma}
For any $t, \tau \in [- T, T]$ (where $T$ is the same as in \eqref{ansatz u}), the flow ${\cal U}_z(\tau, t)$ 
associated to the equation \eqref{NF equation scalare} (with ${\cal U}_z(\tau, \tau) = {\rm Id}$) 
is well defined as a bounded linear operator $H^s \to H^s$ 
and it satisfies 
\begin{equation}\label{stima flusso forma normale BLM}
\| {\cal U}_z(\tau, t) w_0 \|_s \lesssim_s \| w_0 \|_s\,, 
\quad \text{uniformly w.r. to} \quad t , \tau \in [- T, T]\,. 
\end{equation}
for any $w_0 \in H^s$. 
\end{lemma}
\begin{proof}
The proof is exactly the same as the one made in \cite{BLM-growth}, Section 5.1. Indeed the operator ${\rm Op}^{bw}(z)$ is the Weil quantization of the {\it truncated} symbol 
$$
\sigma_z(U; x, \xi) = \sum_{k \in \Z^d} \eta_{\epsilon}\Big(\frac{|k|}{\langle \xi\rangle}\Big) \widehat z(U; k,\xi) e^{\ii k \cdot x}
$$
see \eqref{quantiWeyl}. Since $\widehat \sigma_z(U; k, \xi) = \eta_\e\Big(\frac{|k|}{\langle \xi\rangle}\Big) \widehat z(U; k,\xi)$, one has that also $\sigma_z$ is a symbol in normal form according to the definition \eqref{simboli forma normale}. Hence the arguments developed in \cite{BLM-growth} apply. 
\end{proof}
\noindent
We then denote by 
\[
{\cal U}_{\cal Z}(\tau, t) := \begin{pmatrix}
{\cal U}_z(\tau, t) & 0 \\
0 & \overline{{\cal U}_z(\tau, t)}
\end{pmatrix}\,.
\]
By Duhamel formula, solutions of \eqref{bla bla eq per stima di energia} satisfy 
\begin{equation}\label{W t Duhamel stima energia}
W(t) = {\cal U}_{\cal Z}(0, t) w_0 - \int_0^t {\cal U}_{\cal Z}(\tau, t) {\cal R}(U(\tau)) W(\tau)\, d \tau\,,
\end{equation}
and recall that, by the ansatz on $U(t)$, we have 
$\| U(t) \|_\rho \lesssim \e$,  $\forall t \in [- T, T]$,
for $\rho \gg 0$ large enough and some $T>0$. 
By the estimates \eqref{stima cal R (4)}, \eqref{stima flusso forma normale BLM}, 
one then has if $w_0 \in H^\rho$, 
\[
\| W(t ) \|_\rho 
\lesssim_\rho 
\| W_0 \|_\rho + \int_0^t \| U(\tau) \|_\rho^2 \| W(\tau) \|_\rho\, d \tau 
\lesssim_\rho  
\| W_0 \|_\rho + \e^2 \int_0^t \| W(\tau) \|_\rho\,d \tau\,.
\]
By Gronwall inequality, one then gets that 
\[
\| W(t) \|_\rho \leq C(\rho) e^{C(\rho)  \e^2 t} \| W_0 \|_\rho\,, \quad \forall t \in [- T, T]\,,
\]
for some constant $C(\rho) \gg 0$ large enough. 
By Propositions \ref{prop off diagonal}, \ref{BLM-paradiff}, \ref{BNF step} 
and 
\[
U(t) = \Big(\Phi^{(1)}(U(t)) \circ \Phi^{(2)}(U(t)) \circ \Phi^{(3)}(U(t)) \Big) [W(t)]\,,
\]
one deduces that 
\[
\| U(t) \|_\rho \sim_\rho \| W(t) \|_\rho\,,, \quad \forall t \in [- T, T]
\]
and therefore, 
\[
\| U(t) \|_\rho \leq C_1(\rho) e^{C(\rho)  \e^2 t} \| U_0 \|_\rho\,, 
\quad \forall t \in [- T, T]\,,
\]
for some constant $C_1(\rho) \gg 0$ large enough. 
By a standard bootstrap argument, 
the latter estimate implies that $T = T_\rho = O(\e^{- 2})$ and 
\[
\| U(t) \|_\rho \lesssim_\rho \| U_0 \|_\rho\,, \quad \forall t \in [- T_\rho, T_\rho]\,. 
\]
Clearly, by the smallness assumption on the initial datum 
$\| U_0 \|_\rho \leq \e$, one then gets that $\| U(t) \|_\rho \lesssim_\rho \e$, 
for any $t \in [- T_\rho, T_\rho]$. 
This is the estimate \eqref{stimathm1} in Theorem \ref{teo principale}.

\vspace{0.5em}
We now perform a boothstrap argument in order to show that if 
$s \geq \rho$ and $U_0\in{\bf H}^{s}$ (see \eqref{hcic}) then
\begin{equation}\label{induction hp bootstrap}
\begin{aligned}
&
U \in C^0 \Big( [- T_\rho, T_\rho], H^s \Big) \quad
\text{with}  \quad \| U(t) \|_s \leq C_*(s)\| U_0 \|_s\,, 
\quad  \forall t \in [- T_\rho, T_\rho]\,, \\
& \text{for some} \quad C_*(s) \gg 0 \quad \text{large enough}\,.
\end{aligned}
\end{equation} 
The latter claim implies the estimate \eqref{stimathm2}.
In order to prove the \eqref{induction hp bootstrap}
we argue by induction on $s \geq \rho$. 
If $s = \rho$, then the claimed statement is proved. 
Assume that the statement is true for some $s > \rho$ and let us prove it for $s + 1$. 
Let $U_0 \in {\bf H}^{s + 1}$. 
Then $W_0 \in {\bf H}^{s + 1}$ and by the induction hypothesis 
$U(t)$ and then $W(t)$ is in 
$C^0 \Big( [- T_\rho, T_\rho], {\bf H}^s \Big)$. 
By applying Lemma \ref{NF flow Lemma}, using that the remainder 
${\cal R}(U)$ in \eqref{W t Duhamel stima energia} 
is one smoothing (see \eqref{stima cal R (4)}), 
one has that 
\begin{equation}\label{draghi 0}
\begin{aligned}
& {\cal U}_{\cal Z}(\tau, t) W_0 \in {\bf H}^{s + 1}\,, 
\quad \| {\cal U}_{\cal Z}(\tau, t) W_0\|_{s + 1} \lesssim_s \| W_0 \|_{s + 1}\,, 
\quad \forall \tau, t \in \R\,, 
\\& 
W(\tau) \in {\bf H}^s \;\;\Longrightarrow \;\;{\cal R}(U(\tau))[W(\tau)] \in {\bf H}^{s + 1} \;\;
\Longrightarrow\;\; {\cal U}_{\cal Z}(\tau, t) {\cal R}(U(\tau))[W(\tau)] \in {\bf H}^{s + 1}\,, 
\\& 
\Big\| \int_0^t {\cal U}_{\cal Z}(\tau, t) {\cal R}(U(\tau)) W(\tau)\, d \tau \Big\|_{s + 1} 
\lesssim_s \e^2 \int_0^t \| W(\tau) \|_s\, d \tau\,. 
\end{aligned}
\end{equation}
Therefore, \eqref{W t Duhamel stima energia}, \eqref{draghi 0} imply that 
\begin{equation}\label{draghi 1}
W(t ) \in {\bf H}^{s + 1}\,, \quad 
\| W(t) \|_{s + 1} \lesssim_s \| W_0 \|_{s + 1} + \e^2 \int_0^t \| W(\tau) \|_s \, d \tau\,, 
\quad \forall t \in [- T_\rho, T_\rho]\,. 
\end{equation}
Using that (by the boundedness of the normal form transformations) 
\[
\| U(t) \|_{s + 1} \sim_s \| W(t) \|_{s + 1}, \quad \| U(t) \|_{s } \sim_s \| W(t) \|_{s}\,, 
\quad  \| W_0 \|_{s + 1} \sim_s  \| U_0 \|_{s + 1}\,, 
\quad \| W_0 \|_{s} \sim_s  \| U_0 \|_{s}\,, 
\]
we note that 
\eqref{draghi 1} implies 
\begin{equation*}
U(t ) \in {\bf H}^{s + 1}\,, 
\quad \| U(t) \|_{s + 1} \leq K(s) \Big( \| U_0 \|_{s + 1} + \e^2 \int_0^t \| U(\tau) \|_s \, d \tau \Big)\,, 
\quad \forall t \in [- T_\rho, T_\rho]\,. 
\end{equation*}
for some constant $K(s) \gg 0$ large enough. 
Hence, by the induction hypothesis \eqref{induction hp bootstrap}, 
the latter inequality implies that, for any $t \in[- T_\rho, T_\rho]$,
\begin{equation*}
\| U(t) \|_{s + 1} \leq K(s) \| U_0 \|_{s + 1} + K(s) C_*(s) T_\rho \e^2 \| U_0 \|_s 
\leq C_*(s + 1) \| U_0 \|_{s + 1}
\end{equation*}
with $C_*(s + 1) := K(s) (1 + C_*(s))$ and using that $T_\rho \e^{- 2} \leq 1$. 
The claimed statement \eqref{induction hp bootstrap} 
has then been proved for $s + 1$. 
The proof of Theorem \ref{teo principale} is then concluded.

\appendix
\section{Non-resonance conditions.}\label{sezione non rionanza}
In this sections we verify the non resonance conditions appearing in the Birkhoff normal form. 
We need to provide suitable lower bounds for the {\it three wave interactions}
\begin{equation}\label{divisors BNF}
\begin{aligned}
\phi^{\s,\s'}(\x,k):=\Lambda(\xi + k) +\s \Lambda(\xi) +\s' \Lambda(k)\,,\quad \x,k\in \mathbb{Z}^{d}\,,\;\;\;
\s,\s'\in\{\pm\}\,,
\end{aligned}
\end{equation}
where $\Lambda(\x)$ is the symbol defined in \eqref{def Lambda xi},
for ``most'' choices 
 of the parameter $m\in(0,+\infty)$.
This is the content of the following Lemma.

\begin{lemma}\label{condizioni di non-risonanza}
There exists a set $\mathcal{G}\subseteq (0, + \infty)$ of Lebesgue measure $1$
such that for any $m \in {\cal G}$
there exist $\tau = \tau(d) \geq 0$ and  $\gamma>0$ such that,
for all $\x,k\in \mathbb{Z}^{d}$, $\s,\s'\in\{\pm\}$, one has
\begin{equation}\label{lowerstima}
|\phi^{\s,\s'}(\x,k)|\geq 
\frac{\gamma}{\langle \xi\rangle^\tau \langle k \rangle^\tau}\,.
\end{equation}
\end{lemma}
\noindent
The rest of the section is devoted to the proof of the lemma above.
%
Let us denote by 
\begin{equation}\label{vettOmegaGG}
\omega_g \in \R^{d_*}\,,\quad d_* := \frac{d (d - 1)}{2} + d \,,
\end{equation}
 the vector obtained by putting in a vector the matrix 
 elements (upon the diagonal) of the matrix $G$
 in \eqref{matmetrica}, namely 
 \[
 \omega_g := (g_{11}, \ldots, g_{1 d}, g_{22}, \ldots, g_{2d}, 
 \ldots, g_{(d - 1)(d - 1)}, g_{(d - 1) d}, g_{dd})\,. 
 \]

\noindent
Fix $\tau_* \geq d_*$. We define the set
\begin{equation}\label{diofantea metrica}
{\cal G} := \Big\{ m > 0 : \exists \gamma> 0 \quad \text{such that} \quad  
|\omega_g \cdot \ell \pm  m| \geq 
\frac{\gamma}{\langle \ell \rangle^{\tau_*}}\,, 
\quad 
\forall \ell \in \Z^{d_*}\, \Big\}\,. 
\end{equation}
The following Lemma holds 
\begin{lemma}\label{leb dio m}
The Lebesgue measure of $(0, + \infty) \setminus {\cal G}$ is equal to zero.
\end{lemma}
\begin{proof}
A direct calculation shows that 
\[
\begin{aligned}
& {\cal G}^c := (0,+ \infty) \setminus {\cal G} = 
\cap_{\gamma > 0} \cup_{\ell \in \Z^\nu} {\cal R}_\ell(\gamma), 
\\& 
{\cal R}_\ell(\gamma) := \Big\{ m > 0 : 
|\omega_g \cdot \ell \pm m| < \frac{\gamma}{\langle \ell \rangle^{\tau_*}} \Big\}\,. 
\end{aligned}
\]
Clearly $|{\cal R}_\ell (\gamma)| \lesssim \gamma \langle \ell \rangle^{- \tau_*}$, implying that 
\[
|\cup_{\ell \in \Z^\nu} {\cal R}_\ell(\gamma)| 
\lesssim \sum_{\ell \in \Z^{d_*}} \gamma \langle \ell \rangle^{- \tau_*} 
\lesssim \gamma\,.
\]
This implies that $| {\cal G}^c| = 0$. 
\end{proof}

\begin{proof}[{Proof of Lemma \ref{condizioni di non-risonanza}}]
Clearly, the function $\phi^{\s,\s'}(\x,k)$
 in \eqref{divisors BNF} 
 is very easy to estimate in the case $\s=\s'=+$. Indeed $G$ is positive definite and  therefore 
\[
|\Lambda(\xi + k) + \Lambda(\xi) + \Lambda(k)| \geq 3 m + |\xi + k|^2 + |\xi|^2 + |k|^2
\]
which is bounded away from zero. 

We now  estimate from below $\phi^{\s,\s'}(\x,k)$ in the cases
$\s=-\s'=+$ or $\s=\s'=-$.
Let $\xi, k \in \Z^d$. 
A direct calculation shows that 
\[
\begin{aligned}
\Lambda(\xi + k) - \Lambda(\xi) - \Lambda(k) & = 2 \big\langle G \xi\,,\, k \big\rangle - m 
= 2 \Big( \sum_{i = 1}^d g_{ii} \xi_i k_i 
+ \sum_{i = 1}^d \sum_{j = 1}^{i - 1} g_{i j} (\xi_i k_j + \xi_j k_i)  \Big) - m\,.
\end{aligned}
\]
By the diophantine condition \eqref{diofantea metrica}, one then obtains that for some $\gamma \in (0, 1)$,
\begin{equation}\label{def f xi k}
\begin{aligned}
& |\Lambda(\xi + k) - \Lambda(\xi) - \Lambda(k)| \geq \frac{\gamma}{f(k, \xi)^{\tau_*}} 
\quad \text{where} 
\\& 
f(k, \xi) := 1 + 2 \sum_{i = 1}^d |\xi_i k_i| 
+ 2 \sum_{1 = 1}^d \sum_{j = 1}^{i - 1} |\xi_i k_j + \xi_j k_i|\,. 
\end{aligned}
\end{equation}
We note that
\[
\begin{aligned}
|f(k, \xi)| \lesssim 1 + \sum_{i = 1}^d |\xi_i| |k_i| 
+ \sum_{i, j = 1}^d |\xi_i| |k_j| \leq c(d)\langle \xi \rangle \langle k \rangle
\end{aligned}
\]
for some constant $c(d) \geq 1$.  Hence 
\[
|\Lambda(\xi + k) - \Lambda(\xi) - \Lambda(k)| \geq \frac{\gamma_1}{\langle \xi \rangle^{\tau_*} \langle k \rangle^{\tau_*}} \quad \text{for some} \quad \gamma_1 \ll \gamma\,. 
\]
Similarly, one computes for any $\xi, k \in \Z^d$, 
\[
\begin{aligned}
\Lambda(\xi + k) + \Lambda(\xi) - \Lambda(k) 
& = 2 \| \xi \|_g^2 + 2 \big\langle G \xi, k \big\rangle + m 
\\& 
= 2 \big\langle G \xi\,,\, \xi + k \big\rangle  + m 
\\& 
= 2 \Big( \sum_{i = 1}^d g_{ii} \xi_i (\xi + k)_i 
+ \sum_{i = 1}^d \sum_{j = 1}^{i - 1} g_{i j} (\xi_i (\xi + k)_j + \xi_j (\xi + k)_i)  \Big) + m\,.
\end{aligned}
\]
Hence, using again the diophantine condition \eqref{diofantea metrica} 
and recalling the definition of $f$ in \eqref{def f xi k}, one obtains that 
\[
|\Lambda(\xi + k) + \Lambda(\xi) - \Lambda(k)| \geq \frac{\gamma}{f(\xi + k, \xi)^{\tau_*}}\,. 
\]
Moreover 
\[
\begin{aligned}
f(\xi + k, \xi) & \lesssim 1 + \sum_{i = 1}^d |\xi_i| |\xi_i + k_i| + \sum_{i, j = 1}^d |\xi_i| |\xi_j + k_j|  
\\& 
\lesssim 1 + \sum_{i, j = 1}^d |\xi_i| |\xi_j| 
+ \sum_{i, j = 1}^d |\xi_i| |k_j| \leq c(d)(\langle \xi \rangle \langle k \rangle + \langle \xi \rangle^2)
\end{aligned}
\]
for some constant $c(d) \geq 1$. 
This implies that 
\[
|\Lambda(\xi + k) + \Lambda(\xi) - \Lambda(k)| \geq 
\frac{\gamma_1}{\langle \xi \rangle^{2 \tau_*} \langle k \rangle^{\tau_*}}\,,
\]
and hence Lemma \ref{condizioni di non-risonanza} follows.
\end{proof}

\def\cprime{$'$}



\end{document}